\documentclass[11pt]{article}

\usepackage[a4paper,margin=1in]{geometry}
\usepackage{amsmath,amssymb,amsfonts,amsthm}
\usepackage{mathtools}
\usepackage{mathrsfs}
\usepackage{bbm}
\usepackage{enumerate}
\usepackage{hyperref}
\usepackage{tikz-cd}
\usepackage{ulem}
\usepackage{todonotes}
\usepackage{cancel}

\newtheorem{theorem}{Theorem}[section]
\newtheorem{lemma}[theorem]{Lemma}
\newtheorem{proposition}[theorem]{Proposition}
\newtheorem{corollary}[theorem]{Corollary}

\theoremstyle{definition}
\newtheorem{definition}[theorem]{Definition}
\newtheorem{remark}[theorem]{Remark}
\newtheorem{example}[theorem]{Example}

\numberwithin{equation}{section}

\newcommand{\R}{\mathbb{R}}
\newcommand{\X}{\mathbb{X}}
\newcommand{\Y}{\mathbb{Y}}
\newcommand{\N}{\mathbb{N}}
\newcommand{\E}{\mathbb{E}}
\newcommand{\F}{\mathcal{F}}
\newcommand{\M}{\mathcal{M}}
\newcommand{\B}{\mathcal{B}}
\newcommand{\1}{\mathbbm{1}}

\allowdisplaybreaks

\begin{document}

\title{Strong Measurability and Adapted Approximations with Applications to
Non-Markovian Control}

\author{Keivan Mirzaei\footnotemark[2] \and Jinniao Qiu\footnotemark[2]}
\date{}
\footnotetext[1]{This work was partially supported by the Natural Sciences and Engineering Research Council of Canada (NSERC). }
\footnotetext[2]{Department of Mathematics \& Statistics, University of Calgary, 2500 University Drive NW, Calgary, AB T2N 1N4, Canada. \textit{E-mail}: \texttt{keivan.mirzaei@ucalgary.ca} (K. Mirzaei), \texttt{jinniao.qiu@ucalgary.ca} (J. Qiu).}

\maketitle

\begin{abstract}
  We study strong measurability and adapted approximations for maps taking values in possibly nonseparable 
  metric spaces.
   After recalling the metric-space Pettis
  criterion, we give conditions under which measurability implies strong
  measurability. Under the continuum hypothesis (CH), every measurable map
  from a countably generated measurable space into a metric space is strongly
  measurable for every measure on the domain. For finite or $\sigma$-finite
  measures, the same conclusion holds when the target density is strictly
  smaller than every real-valued measurable cardinal. Under CH this includes
  targets of density at most $\mathfrak c$; if no real-valued measurable
  cardinal exists, it holds for all metric targets.

  For function-valued maps, we show that pointwise measurability and 
  continuity of the sections
need not ensure measurability in the function-space topology.
 We provide a verification
  criterion using countably many evaluations and essential separability of
  the range. Further, we obtain elementary $L^p$-approximations of
  progressive or jointly measurable adapted processes under suitable
  filtration assumptions, Lipschitz functionals with path and Wasserstein
  variables, and smooth approximations based on finitely many observations
  of a process with left- or right-continuous paths. 
  Finally, under CH and suitable assumptions, 
  we apply these results to controlled stochastic differential equations in separable Hilbert spaces, 
  allowing for jumps, random coefficients, path dependence, and state- and control-law dependence;
  we establish convergence of states and costs uniformly over admissible controls, convergence of optimal values, 
  and transfer of near-optimal controls.
\end{abstract}


\section{Introduction}

  In this paper, we consider maps defined on general measure spaces and
  taking values in a possibly non-separable metric space, namely
  \[
    G\colon (\X,\M,\mu) \to (\Y,d).
  \]
  For such functions, the two principal notions considered here are
  measurability and strong measurability; see
  Section~\ref{sec:measurability}. Strong measurability implies measurability
  with respect to the completion of the domain measure; when the domain is
  complete, it therefore implies measurability. Conversely, Borel
  measurability need not imply strong measurability when the target space is
  non-separable. Thus a central question is to identify conditions under which
  ordinary measurability already implies strong measurability.
  Various answers
  to this question appear, either explicitly or implicitly, in the literature on
  functional analysis and measure theory; see, for instance,
  \cite{gorka2026new,MarczewskiSikorski1948,schechter1996handbook}.
  One aim of this paper is to give a unified treatment of conditions under
  which measurability implies strong measurability, highlighting the
  complementary roles of the measurable domain, the underlying measure,
  and the target space.

  The following theorem collects the measurability principles underlying
  our approximation results; see
  Theorems~\ref{thm:pettis-metric}, \ref{thm:first-partial}, and
  \ref{thm:second}, Corollary~\ref{cor:countably-generated}, and
  Remark~\ref{rem:sigma-finite}.
  \begin{theorem}\label{thm-main}
    Let $(\X,\M,\mu)$ be a measure space, let $(\Y,d)$ be a metric space, and let
    \[
      G\colon (\X,\M) \to (\Y,\B(\Y))
    \]
    be measurable. Then:
    \begin{enumerate}[(i)]
      \item $G$ is strongly measurable if and only if it has an
            essentially separable range;
      \item assuming the continuum hypothesis (CH), $G$ is strongly measurable whenever $\M$ is
            countably generated; in particular, this applies when $\X$ is a
            second-countable topological space and $\M$ is its Borel
            $\sigma$-algebra;
      \item if $\mu$ is finite or $\sigma$-finite and the density of $\Y$ is
            strictly smaller than every real-valued measurable cardinal, then
            $G$ is strongly measurable. In particular, under CH this applies
            whenever $\operatorname{dens}(\Y)\leq\mathfrak c$, and it applies
            to every metric space $\Y$ if no real-valued measurable cardinal
            exists.
    \end{enumerate}
  \end{theorem}
  \noindent
These results clarify the relationship between measurability and strong
measurability for functions taking values in possibly non-separable metric
spaces. Assertion~(i) is a \emph{range-side} criterion: a measurable function
is strongly measurable precisely when its range is separable outside a null
set. This criterion may be viewed as a metric-space analogue of Pettis'
measurability theorem (cf.~\cite{Pettis1938}).
Assertion~(ii) gives a \emph{domain-side} condition under which essential
separability of the range is automatic: CH gives this conclusion when the
underlying $\sigma$-algebra is countably generated. Assertion~(iii) instead
combines finiteness or $\sigma$-finiteness of the measure with the sharp
set-theoretic restriction on the density of the target.
 The range-side criterion in assertion~(i) is known in the literature; see, for
example, \cite[Section~21.4]{schechter1996handbook}.
A special case of
assertion~(ii), in which the domain $\mathbb X$ is a compact interval in
$\mathbb R$, appears in \cite{gorka2026new}.
The measure-concentration core of assertion~(iii) is contained in
Theorem~III of Marczewski and Sikorski~\cite{MarczewskiSikorski1948};
see also the equivalent criteria in
\cite[Theorem~3.13]{Pestov2020}.
For finite $\mu$, assertion~(iii) follows from
their theorem by applying it to the push-forward measure and then invoking
assertion~(i). For $\sigma$-finite $\mu$, one first restricts to countably many
finite-measure pieces; the unrestricted push-forward need not be
$\sigma$-finite. We retain the direct proof to make the map-level implication
and its set-theoretic scope explicit in a form suited to the applications below.

Strong measurability plays an important role in the functional-analytic
framework of partial differential equations (PDEs) and stochastic control. Examples include the treatment
of weak and strong measurability in \cite[Appendix~E.5]{Evans2010} and
the Bochner measurability requirements in
\cite[Section~2]{cheungviscosity}. These applications motivate making
explicit when Borel measurability implies strong measurability and how
the relevant hypotheses can be verified for function-valued maps.

A further objective of this paper is to connect strong measurability
with adapted approximation. Related approximation problems arise in
stochastic analysis, PDEs, and stochastic control and games; see,
for instance,
\cite{bayer2022pricing,bayraktar2024deep,cheung2025viscosity,
cosso2024master,jisheng2025wentzell,liang2026viscosity,
qiu2019uniqueness,qiu2023stochastic}.

For function-space-valued maps, we show that pointwise measurability
and continuity of the sections need not ensure Borel measurability
in the function-space topology; see Example~\ref{ex:not-measurable}.
We provide a verification criterion for function-space metrics
determined by countably many evaluations, using essential separability
of the range to obtain a Borel measurable, strongly measurable version
of the lifted map; see Theorem~\ref{thm:measurable-function-valued} and
Corollary~\ref{cor:essential-range-verification}. This criterion
provides a way to verify the measurability assumptions needed for
the function-valued approximation results.

Under the hypotheses specified below, we construct elementary
$L^p$-approximations, $1\le p<\infty$, of metric-space-valued progressive
processes, and of jointly measurable adapted processes when the
filtration satisfies the usual conditions. We also treat Lipschitz
functionals involving path and Wasserstein variables. For metrizable
vector targets, we construct smooth approximations based on finitely
many past observations when the filtration is generated by a process
with left- or right-continuous paths. Finally, under CH and the stated
measurability, Lipschitz, and moment assumptions, we establish stability
under coefficient approximation for path-dependent controlled
McKean--Vlasov equations in separable Hilbert spaces with jumps, random
coefficients and state- and control-law dependence. The state and cost
estimates are uniform over the admissible control class and yield
convergence of optimal values and transfer of near-optimal controls.

The paper is organized as follows. Section~2 presents the measurability
framework and the basic approximation lemmas. Section~3 establishes
strong-measurability criteria and develops the verification principle for
function-valued maps. Section~\ref{sect:applications} develops elementary
and smooth adapted approximations. Section~\ref{sect:nonmarkovian-control}
applies these results to stability of non-Markovian control problems.

%%%%%%%%%%%%%%%%%%%%%%%%%%%%%%%%%%%%%%%%%%%%%%%%%%%%%%%%%%%%%%%%%%
%%%%%%%%%%%%%%%%%%%%%%%%%%%%%%%%%%%%%%%%%%%%%%%%%%%%%%%%%%%%%%%%%%

\section{Preliminaries}

We begin by recalling several standard notions and results concerning topologies, metrics, Borel $\sigma$-algebras, and measurability for functions taking values in Banach and general metric spaces. We also present a concise proof of a variant of the Pettis measurability theorem for metric-space-valued functions, which will be used in the subsequent analysis.

\subsection{Topologies, Metrics, and Borel \texorpdfstring{$\sigma$}{sigma}-Algebras}

Let $\Y$ be a set endowed with two topologies $\tau_1$ and $\tau_2$. We say that $\tau_2$ is \emph{finer} (or \emph{stronger}) than $\tau_1$ if $\tau_1 \subseteq \tau_2$; equivalently, $\tau_1$ is \emph{coarser} than $\tau_2$. Fundamental notions such as measurability, convergence, separability, and continuity of maps into $\Y$ depend on the chosen topology, with finer topologies typically imposing stronger requirements.
In metric spaces, an analogous comparison is given by domination of metrics. Given two metrics $d_1$ and $d_2$ on $\Y$, we say that $d_2$ \emph{dominates} $d_1$ if there exists a constant $C>0$ such that
\[
C\, d_1(y_1,y_2) \le d_2(y_1,y_2), \qquad \forall\, y_1,y_2 \in \Y.
\]
In this case, the topology induced by $d_2$ is finer than that induced by $d_1$.
 
A topological space is called \emph{separable} if it contains a countable dense subset, and \emph{second countable} if it admits a countable base. While second countability is strictly stronger in general, the two notions coincide for metric spaces.
The \emph{Borel $\sigma$-algebra} of a topological space is the smallest $\sigma$-algebra containing all open sets.
In a metric space, the topology is generated by open balls, but the Borel $\sigma$-algebra need not be generated by open balls alone, since $\sigma$-algebras are closed only under countable unions, whereas topologies are closed under arbitrary unions.
This distinction disappears in separable metric spaces: every open set can be written as a countable union of open balls, and hence the Borel $\sigma$-algebra is generated by the open balls. In contrast, this fails in general for non-separable spaces. For example, in an uncountable discrete space, the Borel $\sigma$-algebra coincides with the power set, whereas the $\sigma$-algebra generated by open balls (singletons) consists only of countable sets and their complements, and is therefore strictly smaller.

Such phenomena arise naturally in analysis. Many function spaces are non-separable even when defined over separable domains; for instance, $C_b(\mathbb{R})$ endowed with the supremum norm is not separable.
This necessitates tools from set theory, including transfinite constructions and cardinality arguments.
In particular, we will use the fact that the Borel $\sigma$-algebra of a second-countable topological space has cardinality at most $\mathfrak{c}$. This can be established via a standard transfinite induction argument.
Here $\mathfrak{c}$ denotes the cardinality of $\mathbb{R}$, and the continuum hypothesis (CH) asserts
that no cardinal lies strictly between $\aleph_0$ and $\mathfrak{c}$;
it is independent of Zermelo-Fraenkel set theory with the Axiom of Choice (ZFC).

\subsection{ Weak measurability, measurability, and strong measurability}\label{sec:measurability}

%%%%%%%%%%%%%%%%%%%%%%%%%%%%%%%%%%%%%%%%%%%%%%%%%%%%%%%%%%%%
%%%%%%%%%%%%%%%%%%%%%%%%%%%%%%%%%%%%%%%%%%%%%%%%%%%%%%%%%%%%%

In this section we recall several notions of measurability for Banach-space-valued
and metric-space-valued functions, and also we present a version of Pettis' theorem for metric-space-valued functions.

\begin{definition}\label{def:measurability}
  Let $(\X,\mathcal{M},\mu)$ be a measure space, $(\mathbb B,\|\cdot\|)$ a Banach
  space, and $(\Y,d)$ a metric space. Denote by $\B_\Y$ the Borel
  $\sigma$-algebra induced by the metric topology on $\Y$.
  \begin{enumerate}[(i)]
    \item A function $G\colon \X \to \Y$ is called measurable if it is
          $(\mathcal{M},\B_\Y)$-measurable, i.e.\ $G^{-1}(B) \in \mathcal{M}$ for every
          $B \in \B_\Y$;
    \item A function $G\colon \X \to \Y$ is called \emph{simple} if it is measurable and takes only finitely many values in $\Y$, that is, there
          exist $y_1,\dots,y_n \in \Y$ and a partition $E_1,\dots,E_n \in \mathcal{M}$
          of $\X$ such that
          \[
            G(x) = \sum_{i=1}^n y_i\,\1_{E_i}(x);
          \]
          Observe that for the summation above to be well-defined, we do not
          need to assume any algebraic structure on $\Y$. To interpret the sum,
          we simply define $G(x)$ to be $y_i$ if $x \in E_i$.
    \item A function $G\colon \X \to \Y$ is called \emph{strongly measurable} if it is the $\mu$-almost everywhere limit of a
          sequence of simple functions;
    \item A function $G\colon \X \to \mathbb B$ is called \emph{weakly measurable} if for every continuous linear functional
          $f \in \mathbb B^\ast$ the scalar-valued function $f \circ G\colon \X \to \R$ is
          measurable. In other words, the following diagram commutes for every $f \in \mathbb B^\ast$:
          \[
            \begin{tikzcd}[row sep=large, column sep=large]
              (\X, \mathcal{M}) \arrow[r, "G"] \arrow[dr, "f\circ G"'] & \mathbb B \arrow[d, "f"] \\
              & (\mathbb{R}, \B_{\mathbb{R}})
            \end{tikzcd}
          \]
  \end{enumerate}
\end{definition}

Clearly, simple functions are measurable, and measurable functions are weakly
measurable because continuous linear functionals are measurable.
A fundamental result on the equivalence of weak measurability and strong measurability is the Pettis measurability theorem (see \cite{Pettis1938}), which states that
    a weakly measurable function $G\colon (\mathbb X,\mathcal{M},\mu) \to (\mathbb B,\|\cdot\|)$  is strongly
    measurable if and only if there exists a
    $\mu$-null set $N \in \mathcal{M}$ such that $G(X \setminus N)$ is a separable
    subset of $\mathbb B$.
As a consequence, if $\mathbb B$ is separable, then every measurable function
$G \colon \X \to \mathbb B$ is strongly measurable, since $G(\X)$ is contained in the
separable space $\mathbb B$ and measurability implies weak measurability. When $\mathbb B$ is
non-separable, however, measurable functions need not be strongly measurable, as
the next example shows.
\begin{example}\label{ex:non-strong}
  Let $\mathbb B$ be a \emph{non-separable} Banach space and let $\B$ be the Borel
  $\sigma$-algebra on $\mathbb B$. Consider the counting measure $\kappa$ on the measurable
  space $(\mathbb B,\mathcal{P}(\mathbb B))$, where $\mathcal{P}(\mathbb B)$ denotes the power set of
  $\mathbb B$.
  Define the identity map
  \[
    i\colon (\mathbb B,\mathcal{P}(\mathbb B),\kappa) \to (\mathbb B,\B), \qquad i(x) = x.
  \]
  This map is continuous and therefore measurable. On the other hand, it is not
  strongly measurable. Indeed, the only $\kappa$-null set in the domain is the
  empty set, and the essential range of $i$ is all of $\mathbb B$, which is non-separable
  by assumption. Thus $i(\mathbb B)$ is not a separable subset of $\mathbb B$, and
  Pettis' measurability theorem shows that $i$ cannot be strongly measurable.
\end{example}

We now recall a metric-space analogue of Pettis' measurability theorem, which
will be used in the arguments that follow. In many applications,
the processes under consideration take values in possibly non-separable {metric
spaces}, rather than Banach spaces. The formulation below is therefore more
directly suited to our setting, and it can be stated without invoking weak
measurability. Although this result has already been known in the literature (see \cite[Section 21.4]{schechter1996handbook} for instance), we include a short proof for completeness and for the reader's sake.

\begin{theorem}\label{thm:pettis-metric}
  Let $(\Y,d)$ be a metric space and let
  $G\colon (\X,\mathcal{M})\to(\Y,\mathcal{B}(\Y))$ be measurable. Then $G$ is
  strongly measurable with respect to $\mu$ if and only if it has an
  \emph{essentially separable} range, i.e., there exists a $\mu$-null set
  $N\in\mathcal M$ such that $G(\X\setminus N)$ is contained in a separable
  subset of $\Y$.
\end{theorem}

\begin{proof}
  For the ``only if" part, the closure of the union of the finite ranges of
  a sequence of simple approximants is separable and contains $G(x)$ wherever
  the approximants converge to $G(x)$.
  For the ``if" part, assume $G$ is measurable and there exists a null set $N$ such that
  $G(\X \setminus N)$ is a separable subspace of $\Y$.
  If $\X\setminus N$ is empty, the conclusion is immediate.
  Since $G(\X \setminus N)$ is separable, there exists a countable dense subset
  $\{y_n:\, n\in \mathbb N\}$ of $G(\X \setminus N)$.
 
  For every $m\in\N$, and for every $x \in \X$, define $G_m(x)$ to be
  the first element in the finite set $\{y_1,\dots,y_m\}$ which is closest to $G(x)$.
  To see that $G_m$ is measurable, observe that for every $n\leq m$,
  \[
    \begin{aligned}
      E_{m,n}={}&\{y\in\Y: d(y,y_n)\leq d(y,y_j)\text{ for all }j\leq m\}\\
      &\cap\{y\in\Y: d(y,y_n)<d(y,y_j)\text{ for all }j<n\}
    \end{aligned}
  \]
  is a Borel set (because the functions $d(\cdot, y_j)$ are continuous),
  and hence the sets $G^{-1}(E_{m,n})$, $n\leq m$, form a measurable partition
  on which $G_m$ is constant. Thus $G_m$ is simple.
  Now observe that for every $x \in \X \setminus N$,
  $d(G(x), G_m(x)) = \min_{1 \leq j \leq m} d(G(x), y_j)$,
  and since $\{y_n:\, n\in \mathbb N\}$ is dense in $G(\X \setminus N)$,
  we have $d(G(x), G_m(x)) \to 0$ as $m \to \infty$.
\end{proof}

On an incomplete measure space, a strongly measurable map $G$ need not itself
be measurable, but it has a measurable version. Indeed, choose a measurable
null set $N$ outside which its simple approximants converge to $G$, and replace
both $G$ and the approximants on $N$ by a fixed value in $\Y$. The modified map
is an everywhere pointwise limit of measurable simple functions, hence is
measurable. In integrals involving strongly measurable maps, we use the
completion of $\mu$; replacing a map by this measurable version does not change
these integrals.

Strong measurability is intimately connected to the approximation of functions by simple functions.
The following lemma provides a tool for simplifying approximation arguments in subsequent sections.
\begin{lemma}\label{lem:monotone-approximation}
  Let $(\Y, d)$ be a metric space.
  Suppose the function $G\colon (\X,\M,\mu) \to (\Y,\B(\Y))$ is strongly measurable.
  Then there exists a sequence $\{G_n:\, n\in \N\}$ of simple functions such that, for $\mu$-almost every $x \in \mathbb X$, the sequence
  $ d(G_n(x), G(x))$ is monotonically decreasing to $0$.
\end{lemma}
\begin{proof}
  Replace $G$ by its measurable version as described above; this does not affect
  the required almost-everywhere conclusion.
  By the definition of strong measurability, there exists a sequence
  $\{\tilde{G}_n:\, n\in \N\}$ of simple functions such that $ d(\tilde{G}_n(x), G(x)) \to 0$
  as $n\to \infty$ for $\mu$-almost every $x \in \X$.
  Define the sequence $\{G_n:\, n\in \N\}$ as follows:
  \[
    G_1(x) := \tilde{G}_1(x), \qquad
    G_n(x) :=
    \begin{cases}
      G_{n-1}(x),     & \text{if } d(\tilde{G}_n(x), G(x)) \geq  d(G_{n-1}(x), G(x)), \\
      \tilde{G}_n(x), & \text{otherwise}.
    \end{cases}
  \]
  The comparison sets are measurable, so each $G_n$ is simple. For
  $\mu$-almost every $x \in \X$, the sequence
  $ d(G_n(x), G(x))$ is monotonically decreasing to $0$.
\end{proof}

A straightforward consequence of the strong measurability and monotone
approximation lemma is the following
approximation result for metric-space-valued functions.

\begin{lemma}[An approximation lemma]\label{lem:approx}
  Let $(\Y,d)$ be a metric space, $(\X,\mathcal{M},\mu)$ an arbitrary measure
  space, $1 \le p < \infty$, and let
  \[
    G\colon (\X,\mathcal{M}) \to (\Y,\B_\Y)
  \]
  be \emph{strongly measurable} and there exists $y_0\in \mathbb Y$ such that $\int_\X d(G(x), y_0)^p \,\mu(dx) < \infty$. Then for every $\varepsilon>0$ there exist
  $y_1,\dots,y_n \in \Y$ and a partition $E_1,\dots,E_n \in \mathcal{M}$ such
  that the simple function
  \[
    \widetilde{G}(x) := \sum_{k=1}^n y_k \,\1_{E_k}(x)
  \]
  satisfies
  \[
    \int_\X d(G(x), \widetilde{G}(x))^p \,\mu(dx) \le \varepsilon.
  \]
\end{lemma}
\noindent
  Before proceeding to the proof, we make some remarks:
  \begin{itemize}
    \item Observe that the measure space $(\X,\mathcal{M},\mu)$ is arbitrary and
          need not be finite or $\sigma$-finite.
    \item It is clear that the approximation result does not hold for $p=\infty$.
          Therefore, the approximation is not uniform in $x$, but only in
          average with respect to $\mu$.
  \end{itemize}

\begin{proof}[Proof of Lemma~\ref{lem:approx}]
  Again replace $G$ by its measurable version, which preserves the integrals
  in the statement.
  According to the monotone approximation Lemma~\ref{lem:monotone-approximation}
  and the fact that $G$ is strongly measurable, there exists a sequence
  $\{\bar{G}_n:\, n\in \N\}$ of simple functions such that the sequence
  $d(\bar{G}_n(x), G(x))$
  decreases monotonically to $0$ as $n\to \infty$ for $\mu$-almost every
  $x \in \X$. Set $\tilde{G}_n(x) := \bar{G}_n(x)$ if $d(\bar{G}_n(x), G(x)) \le d(G(x), y_0)$, and
  $\tilde{G}_n(x) := y_0$ otherwise. Then $\tilde{G}_n$ is also a simple function for each $n$, and
  \[
    d(\tilde{G}_n(x), G(x)) \le d(\bar{G}_n(x), G(x)) \quad\text{and}\quad
    d(\tilde{G}_n(x), G(x)) \le d(G(x), y_0) \quad\text{for all } x \in \X.
  \]
  By the dominated convergence theorem
  \[
    \lim_{n \to \infty} \int_\X d(\tilde{G}_n(x), G(x))^p \,\mu(dx) = 0,
  \]
  where the convergence is actually monotone. Hence, for $n$ large enough we have
  \[
    \int_\X d(\tilde{G}_n(x), G(x))^p \,\mu(dx) \le \varepsilon.
  \]
  This proves the desired conclusion.
\end{proof}

%%%%%%%%%%%%%%%%%%%%%%%%%%%%%%%%%%%%%%%%%%%%%%%%%%%%%%%%%%%%%
%%%%%%%%%%%%%%%%%%%%%%%%%%%%%%%%%%%%%%%%%%%%%%%%%%%%%%%%%%%%%

\section{The equivalence of strong measurability and measurability}

In this section, we prove two results identifying conditions under which
measurability implies strong measurability. In contrast to
Theorem~\ref{thm:pettis-metric}, a metric-space analogue of Pettis'
measurability theorem in which separability is imposed on the \emph{range},
the results below identify assumptions under which essential separability is
automatic. More precisely, the first result uses separability of the domain, or
countable generation of the relevant $\sigma$-algebra, while the second combines
$\sigma$-finiteness of the measure with a density condition on the target. These results are presented in the first
two subsections, followed by a discussion of the limitations and a resolution of a related implementation issue.

\subsection{Separable domain and countably generated sigma-algebras}
The first result establishes that, under the continuum hypothesis, every
measurable mapping from a separable measurable domain---or, more generally,
from a domain equipped with a countably generated $\sigma$-algebra---into a metric
space is strongly measurable with respect to an arbitrary measure. The proof
relies on a cardinality argument together with the continuum hypothesis.

\begin{theorem}[First result]\label{thm:first-partial}
  Given a second-countable topological space $(\X,\tau)$ equipped with  Borel $\sigma$-algebra $\mathcal{B}_\X$, and a metric space $(\Y,d)$ endowed  with
  Borel $\sigma$-algebra $\B_\Y$, suppose the mapping $G\colon (\X,\mathcal{B}_\X) \to (\Y,\B_\Y)$ is
  measurable. Then under $CH + ZFC$, $G$ is strongly measurable with respect to \emph{every}
  measure $\mu$ on $(\X,\mathcal{B}_\X)$.
\end{theorem}

\begin{proof}
  Suppose, for a contradiction,
  that there exists a measure $\mu$ on $(\X,\mathcal{B}_\X)$ such that $G$ is
  not strongly measurable. Then, by the Pettis measurability theorem
  (Theorem~\ref{thm:pettis-metric}), for every $\mu$-null set $N \in \mathcal{B}_\X$ the set
  $G(\X \setminus N)$ is non-separable. In particular, $G(\X)$ is a non-separable
  subset of $\Y$ (taking $N=\varnothing$).
\\[4pt]
  A standard argument using Zorn's lemma and the fact that $G(\X)$ is a non-separable
  subset of $\Y$ shows that there exist $\varepsilon>0$
  and an uncountable subset $A \subset \X$ such that
  \[
    d(G(x), G(y)) \ge \varepsilon \qquad\text{for all distinct } x,y \in A.
  \]
  For each $y \in A$ let $U_y$ denote the open ball in $\Y$ of radius $\varepsilon/2$
  centered at $G(y)$, and for every $E \subset A$ define the open sets
  \[
    U_E := \bigcup_{y \in E} U_y.
  \]
  Then the family $\{G^{-1}(U_E) : E \subset A\}$ consists of
  \textit{distinct} elements of $\mathcal{B}_\X$ because for every $E \subset A$,
  \[
    G^{-1}(U_E) \cap A = E.
  \]
  because $U_y$ contains $G(y)$ but is at least a distance $\varepsilon/2$ away
  from any $G(z)$ with $z\in A , z \neq y$.
  Thus, the family $\{ U_E : E \subset A\}$ has the same cardinality as the power set
  $\mathcal{P}(A)$. Now observe that
  \[
    |\mathcal B_\X| \geq |\mathcal P(A) | = 2^{|A|} \geq 2^{\aleph_1} > 2^{\aleph_0} = \mathfrak c.
  \]
  Therefore, the cardinality of the Borel $\sigma$-algebra $\mathcal B_\X$ is
  strictly greater than the cardinality of the continuum $\mathfrak c$, which
  obviously contradicts the fact the cardinality of a second-countable
  topological space's Borel $\sigma$-algebra is at most $\mathfrak c$.
\end{proof}

The preceding argument applies to an arbitrary measurable space
$(X,\mathcal{M})$ provided that $\mathcal{M}$ is countably generated.
In particular, the proof does not rely
on any topological structure on $X$. We thus obtain the following corollary.

\begin{corollary}\label{cor:countably-generated}
  Let $(\X,\mathcal{M})$ be a measurable space such that
  $\mathcal{M}$ is generated by a countable collection of sets. Let
  $(\Y,d)$ be a metric space with its Borel $\sigma$-algebra $\B_\Y$, and
  $G\colon (\X,\mathcal{M}) \to (\Y,\B_\Y)$ measurable. Then under $CH + ZFC$,
  $G$ is strongly
  measurable with respect to \emph{every} measure $\mu$ on $(\X,\mathcal{M})$.
\end{corollary}

  An important special case of a countably generated $\sigma$-algebra is the
  Borel $\sigma$-algebra on a separable metric space, which is second-countable.
  Therefore, the proof of the preceding theorem implies that the image of a separable
  metric space under a measurable mapping into a metric space is always
  separable under CH.
\begin{example}
  Assume CH. Let $\Y$ be any metric space and let $(\Omega,\mathcal{F}) := (C[0,T],\B_{C[0,T]})$,
  where $C[0,T]$ is endowed with the supremum norm topology. Then $C[0,T]$ is a
  separable metric space as the interval $[0,T]$ is compact. Hence, by Theorem \ref{thm:first-partial}, any measurable random element
  \[
    G\colon (\Omega,\mathcal{F}) \to (\Y,\B_\Y)
  \]
  is strongly measurable with respect to every (probability) measure $\mathbb P$ on
  $(\Omega,\mathcal{F})$. Furthermore, consider $\Omega \times [0,T]$ equipped with
  the predictable $\sigma$-algebra $\mathcal{P}\subseteq \F \otimes \B_{[0,T]}$
  corresponding to some filtration $(\F_t)_{t \in [0,T]}$ on $(\Omega,\F)$.
  Since $\F \otimes \B_{[0,T]}$ is countably generated,
  $|\mathcal P|\leq |\F \otimes \B_{[0,T]}|\leq\mathfrak c$.
  The cardinality argument in the proof of Theorem~\ref{thm:first-partial}
  therefore applies directly to $\mathcal P$, and any
  $\Y$-valued predictable process
  \[
    G\colon (\Omega \times [0,T],\mathcal{P}) \to (\Y,\B_\Y)
  \]
  is strongly measurable with respect to the product measure $\mathbb P \otimes \lambda$ on
  $(\Omega \times [0,T],\mathcal{P})$, where $\lambda$ is the Lebesgue measure on $[0,T]$.
\end{example}

\subsection{Real-valued measurable cardinals and revisiting the equivalence of strong measurability and measurability in \texorpdfstring{$\sigma$}{sigma}-finite measure spaces}

The second result concerns the role of finiteness, and more generally
$\sigma$-finiteness, of the measure on the domain. Its measure-theoretic core is
Theorem~III of Marczewski and Sikorski~\cite{MarczewskiSikorski1948}, developed
in connection with Banach's generalized problem of measure. In current
terminology, their theorem gives a cardinal condition under which a
$\sigma$-finite Borel measure on a metric space is concentrated on a separable
subspace. Combined with
Theorem~\ref{thm:pettis-metric}, this yields a criterion for the equivalence of
measurability and strong measurability. We formulate and prove that criterion
directly for measurable maps, a form particularly suited to metric-space-valued
random fields and stochastic processes.
\\[4pt]

We now introduce a set-theoretic notion which plays a crucial role in our
second result on strong measurability.

\begin{definition}
A cardinal \(\kappa\) is called \emph{real-valued measurable} if there exist a
set \(I\) of cardinality \(\kappa\) and a countably additive probability measure
\(\pi\) defined on every subset of \(I\) such that:
\begin{enumerate}
    \item \(\pi(\{x\})=0\) for every \(x\in I\); and
    \item whenever fewer than \(\kappa\) subsets of \(I\) each have measure
    zero, their union also has measure zero.
\end{enumerate}

The \emph{density} \(\operatorname{dens}(Y)\) of a metric space \(Y\) is the
smallest cardinality of a dense subset of \(Y\).
\end{definition}
 
A standard consequence of this definition is that a set $I$ admits a countably
additive probability measure on $\mathcal P(I)$ which vanishes on singletons if
and only if $|I|$ is at least some real-valued measurable cardinal; see
\cite[Section~3.2]{Pestov2020}. In particular, no such measure exists when
$|I|$ is strictly smaller than every real-valued measurable cardinal. Moreover,
for a metric space $\Y$, every Borel probability measure on $\Y$ is concentrated
on a separable subspace if and only if $\operatorname{dens}(\Y)$ is strictly
smaller than every real-valued measurable cardinal
\cite[Theorem~3.13]{Pestov2020}.
 
Banach and Kuratowski~\cite{BanachKuratowski1929} proved that CH precludes a
countably additive probability measure on $\mathcal P(\R)$ which vanishes on
singletons. This concerns cardinality $\mathfrak c$ and does not rule out
real-valued measurable cardinals larger than $\mathfrak c$. Since every
real-valued measurable cardinal is weakly inaccessible, CH nevertheless implies
that every cardinal at most $\mathfrak c$ is strictly smaller than every
real-valued measurable cardinal; see \cite[Introduction, p.~3]{Pestov2020}.
Consequently, under CH the density hypothesis below holds whenever
$\operatorname{dens}(\Y)\leq\mathfrak c$. This includes almost all metric spaces
of interest in applications.

\begin{lemma}[Theorem~4.4.3 in Engelking \cite{Engelking}]\label{lem:sigma-disjoint}
  Every metrizable space admits a $\sigma$-disjoint base, i.e.\ a base which is the
  union of countably many families of pairwise disjoint open sets.
\end{lemma}
 
We now use these facts to obtain a second result on strong measurability for
finite-measure spaces.

\begin{theorem}[Second result]\label{thm:second}
  Let $(\X,\mathcal{M},\mu)$ be a \emph{finite} measure space, $(\Y,d)$ a metric
  space, and $G\colon (\X,\mathcal{M}) \to (\Y,\B_\Y)$ a measurable function.
  \begin{enumerate}[(i)]
    \item If $\operatorname{dens}(\Y)$ is strictly smaller than every
          real-valued measurable cardinal, then $G$ is strongly measurable.
    \item If a real-valued measurable cardinal exists, then there exist a
          finite measure space $(\X,\mathcal{M},\mu)$, a non-separable Banach
          space $(\Y,\|\cdot\|)$, and a measurable function
          $G\colon (\X,\mathcal{M}) \to (\Y,\B_\Y)$ which is \emph{not} strongly
          measurable.
  \end{enumerate}
\end{theorem}

\begin{proof}
  We first prove (i). Let $\nu$ be the pushforward measure of $\mu$ under $G$,
  i.e.\ for every $B \in \B_\Y$,
  \[
    \nu(B) := \mu(G^{-1}(B)).
  \]
  Since $\mu$ is finite, $\nu$ is a finite measure on $(\Y,\B_\Y)$.
 
  If $\Y$ is separable, the conclusion follows directly from
  Theorem~\ref{thm:pettis-metric}. We may therefore suppose that $\Y$ is
  non-separable and fix a dense set $D\subseteq\Y$ with
  $|D|=\operatorname{dens}(\Y)$.
 
  According to Lemma~\ref{lem:sigma-disjoint}, $\Y$ admits a
  $\sigma$-disjoint base $\mathcal{U} = \bigcup_{n=1}^\infty \mathcal{U}_n$, which
  means that for each $n \in \N$ the elements of $\mathcal{U}_n$ are pairwise
  disjoint and $\mathcal{U}$ forms a base for the topology of $\Y$.
 
  Observe that
  \[
    \{ U \in \mathcal{U} : \nu(U) > 0\} = \bigcup_{n,m=1}^\infty
    \{ U \in \mathcal{U}_n : \nu(U) > 1/m\},
  \]
  and $\{ U \in \mathcal{U}_n : \nu(U) > 1/m\}$ is finite for each $n \in \N$,
  because $\nu$ is finite. Hence the set
  \[
    \{ U \in \mathcal{U} : \nu(U) > 0\}
  \]
  is countable.
 
  For each $n \in \N$ let $\mathcal{V}_n$ be the collection of all sets
  $U \in \mathcal{U}_n$ such that $\nu(U)=0$, and define
  \[
    M_n := \bigcup_{U \in \mathcal{V}_n} U.
  \]
  Then $M_n$ is an open (hence measurable) subset of $\Y$. We claim that
  $\nu(M_n)=0$ for every $n$.
 
  Suppose, on the contrary, that $\nu(M_n)>0$ for some $n$. For every
  $\mathcal E\subseteq\mathcal V_n$, the union $\bigcup_{U\in\mathcal E}U$ is
  open, so we may define
  \[
    \widetilde{\nu}(\mathcal E)
    :=\frac{\nu\bigl(\bigcup_{U\in\mathcal E}U\bigr)}{\nu(M_n)},
    \qquad \mathcal E\subseteq\mathcal V_n.
  \]
  Since the members of $\mathcal V_n$ are pairwise disjoint,
  $\widetilde{\nu}$ is a countably additive probability measure on
  $\mathcal P(\mathcal V_n)$. Moreover,
  $\widetilde{\nu}(\{U\})=0$ for every $U\in\mathcal V_n$.
 
  Each nonempty open set $U\in\mathcal V_n$ meets $D$. Choosing one point of
  $D\cap U$ for each $U\in\mathcal V_n$ gives an injection
  $\mathcal V_n\to D$, because the members of $\mathcal V_n$ are disjoint.
  Hence
  \[
    |\mathcal V_n|\le |D|=\operatorname{dens}(\Y).
  \]
  The existence of $\widetilde{\nu}$ implies that $|\mathcal V_n|$ is at least
  some real-valued measurable cardinal, contrary to the hypothesis in part~(i).
  Therefore, $\nu(M_n)=0$ for every $n\in\N$.
 
  Now let
  \[
    N := \bigcup_{n=1}^\infty G^{-1}(M_n).
  \]
  Then $\mu(N) = \nu\Bigl(\bigcup_{n=1}^\infty M_n\Bigr) = 0$. Moreover,
  \[
    G(\X \setminus N) \subseteq \Y_0 := \Y \setminus \bigcup_{n=1}^\infty M_n.
  \]
  By construction, $\Y_0$ has a countable base: for each $n$, all elements of
  $\mathcal{U}_n$ with positive $\nu$-measure form a countable family, and they
  cover $\Y_0$ when combined over $n$. Thus $\Y_0$ is a separable metric space. By
  Pettis' theorem, $G$ is strongly measurable as desired.
  This proves (i).
 
  For (ii), let $\kappa$ be a real-valued measurable cardinal and choose a set
  $\X$ of cardinality $\kappa$ together with a witnessing probability measure
  \[
    \pi\colon \mathcal{P}(\X) \to [0,1]
  \]
  which is countably additive and satisfies $\pi(\{x\})=0$ for all $x\in\X$.
  Let $\Y := \ell^\infty(\X)$ be the space of bounded real-valued functions on $\X$,
  equipped with the supremum norm $\|\cdot\|_\infty$. Then $\Y$ is a non-separable
  Banach space because $\X$ is uncountable.
\\[4pt]
  Define
  \[
    G\colon (\X,\mathcal{P}(\X)) \to (\Y,\B_\Y), \qquad
    G(x) := \1_{\{x\}},
  \]
  viewed as a function in $\ell^\infty(\X)$. The map $G$ is clearly measurable.
  However, it is not strongly measurable. Indeed, for any $\pi$-null set
  $N \subset \X$ we have that $\X \setminus N$ has positive $\pi$-measure and is
  therefore uncountable. Moreover, for distinct $x,y \in \X$,
  \[
    \|G(x) - G(y)\|_\infty = 1,
  \]
  so $G(\X \setminus N)$ is an uncountable set whose points are pairwise distant by
  $1$. As a consequence, $G(\X \setminus N)$ is not a separable subset of $\Y$, and
  Pettis' theorem implies that $G$ is not strongly measurable. This completes the
  proof.
\end{proof}

\begin{remark}\label{rem:sigma-finite}
  The assumption that $\mu$ is finite in Theorem~\ref{thm:second}(i) can be
  weakened to $\sigma$-finiteness. Indeed, if $\mu$ is $\sigma$-finite, then there
  exist disjoint measurable sets $(\X_n)_{n=1}^\infty$ such that
  $\X = \bigcup_{n=1}^\infty \X_n$ and $\mu(\X_n) < \infty$ for each $n$. If
  $G\colon (\X,\mathcal{M}) \to (\Y,\B_\Y)$ is measurable, then its restriction
  $G_n := G|_{\X_n}$ is measurable on $(\X_n,\mathcal{M}|_{\X_n})$ for each $n$.
  Applying part~(i) of the theorem, we see that each $G_n$ is strongly measurable.
  Hence there exists a null set $N_n \in \mathcal{M}|_{\X_n}$ such that
  $G_n(\X_n \setminus N_n)$ is separable. Let $N := \bigcup_{n=1}^\infty N_n$. Then
  $\mu(N)=0$ and
  \[
    G(\X \setminus N) = \bigcup_{n=1}^\infty G_n(\X_n \setminus N_n)
  \]
  is a countable union of separable sets, hence separable. Pettis' theorem again
  implies that $G$ is strongly measurable.
 
  Thus part~(i) holds without change for $\sigma$-finite measures. If no
  real-valued measurable cardinal exists, its density hypothesis is automatic
  for every metric target. If such a cardinal exists, part~(ii) gives a
  counterexample on a probability space. Under CH, part~(i) applies in
  particular whenever $\operatorname{dens}(\Y)\le\mathfrak c$, but CH alone
  does not settle the assertion for arbitrary metric targets.
\end{remark}

%%%%%%%%%%%%%%%%%%%%%%%%%%

\subsection{On verifying measurability assumptions}\label{sect:verifying-measurability}

The equivalence above is primarily a characterization rather than a verification
tool: for function-valued maps, checking Borel measurability can be as delicate
as checking strong measurability. The purpose of this subsection is to reverse
the usual order of verification. We first obtain measurability of inverse images
of balls from countably many scalar evaluations. Separability of the actual
range then upgrades this ball-measurability to Borel measurability and hence to
strong measurability. Essential separability gives the same conclusion up to a
null-set modification. This provides a practical route to the approximation
results without requiring direct verification of inverse images of arbitrary
open subsets of a non-separable function space.
 
In many situations, we are given a functional $G$ defined on the product of a
measurable space $(\mathbb X,\mathcal A, \mu)$ and a metric space
$(\mathbb Y,d)$, and it has a certain regularity in the
second variable. We want to use our approximation results to approximate the
function $G$. In order to do this, we first push the second variable of $G$ into
the codomain because our approximation results are stated for functions defined
on a measure space and taking values in a metric space. More precisely, we define
\begin{equation}\label{eq:Gtilde}
\widetilde G:\mathbb X\to \mathcal F(\Y) \subseteq\mathbb R^{\Y},
\qquad
\widetilde G(x) = G(x,\cdot),
\end{equation}
where $\mathcal F(\Y)$ is a suitable function space of functions on $\Y$ taking
values, for example, in $\R$.
 
Next, we need to equip $\mathcal F(\Y)$ with a metric $\widetilde d$. At this step,
we have a lot of freedom in choosing the metric $\widetilde d$ on $\mathcal F(\Y)$.
For example, we can choose $\widetilde d$ to be the supremum metric if
$\mathcal F(\Y)$ in \eqref{eq:Gtilde} consists of bounded functions. The trade-off
is that choosing a strong metric $\widetilde d$ on $\mathcal F(\Y)$ will make it
harder to verify the measurability of $\widetilde G$ and to apply the
approximation results, while choosing a weak metric $\widetilde d$ on
$\mathcal F(\Y)$ might make the approximation results less effective.
 
The following example shows that the measurability of $\widetilde G$ is not
automatic, even if $G$ is measurable in the first variable and continuous in the
second variable.

\begin{example}\label{ex:not-measurable}
For every $A\subseteq \mathbb N$, define the piecewise affine continuous function
\[
f_A(y):=\sum_{n\in A}\max\{1-4|y-n|,0\},\qquad y\in\mathbb R.
\]
Then $f_A:\mathbb R\to[0,1]$, $f_A(n)=\1_A(n)$ for every $n\in\mathbb N$,
and its support is the union of the intervals $[n-1/4,n+1/4]$ for $n\in A$.
Now let $\mathbb X = [0,1)$ with the Borel $\sigma$-algebra
$\mathcal B$, $\mathbb Y = \mathbb R$ with the usual metric. Fix $x \in [0,1)$,
and let $(x_1x_2x_3\cdots)_2$ be the binary expansion of $x$ with the convention
that the expansion does not end with infinitely many $1$'s. Let
$A_x= \{n\in\mathbb N : x_n = 1\}$, and define
\[
G:\mathbb X \times \mathbb Y \to \R,
\qquad
G(x,y) = f_{A_x}(y).
\]
Then $G$ is continuous in the second variable by construction, and it is
measurable in the first variable since if we fix $y\in\mathbb Y$, then
$G(\cdot,y)$ can only take two values, and is non-zero only if $y$ is in an
interval of the form $[n-1/4,n+1/4]$ for some $n\in\mathbb N$, and $x_n = 1$.
However, the function $\widetilde G$ defined by \eqref{eq:Gtilde} into
$C_b(\mathbb R)$ with the supremum metric is not measurable.
This is because the range of $\widetilde G$ has the discrete topology under the
supremum metric, and the map $x\mapsto \widetilde G(x)$ is injective, so
$\widetilde G$ is not measurable since $[0,1)$ has non-measurable subsets.
\end{example}

The example above shows that the measurability of $\widetilde G$ is not automatic,
but an interesting fact is that the preimage of every open ball in $C_b(\R)$ with
the supremum metric under $\widetilde G$ is measurable. We can generalize this
observation to the case where $\mathcal F(\Y)$ is a general function space equipped
with a metric $\widetilde d$ that can be reconstructed from countably many
evaluations. This is the content of the following theorem.

\begin{theorem}\label{thm:measurable-function-valued}
Let $\mathbb Y$ be a set and let $\mathcal F(\Y) \subset \mathbb R^{\mathbb Y}$
be a collection of functions. Suppose $\mathcal F(\Y)$ is equipped with a metric
$\widetilde d$, and denote by $\mathcal B(\mathcal F(\Y))$ the associated Borel
$\sigma$-algebra.
Assume that there exists a sequence $(y_n)_{n\ge1}$ in $\mathbb Y$ and a Borel
measurable map
\[
\rho:\mathbb R^{\mathbb N}\times \mathbb R^{\mathbb N}\to [0,\infty]
\]
such that, on defining
\[
\Phi:\mathcal F(\Y)\to \mathbb R^{\mathbb N},
\qquad
\Phi(f) = \bigl(f(y_n)\bigr)_{n\ge1},
\]
we have
\[
\widetilde d(f,g) = \rho\bigl(\Phi(f),\Phi(g)\bigr)
\qquad\text{for all }f,g\in\mathcal F(\Y).
\]
Let $(\mathbb X,\mathcal A)$ be a measurable space, and let
\[
G:\mathbb X\times \mathbb Y \to \mathbb R
\]
be such that:
\begin{enumerate}[(i)]
    \item for every $y\in\mathbb Y$, the map $x\mapsto G(x,y)$ is $\mathcal A$-measurable;
    \item for every $x\in\mathbb X$, the function $G(x,\cdot)$ belongs to $\mathcal F(\Y)$.
\end{enumerate}
Define
\[
\widetilde G:\mathbb X\to \mathcal F(\Y),
\qquad
\widetilde G(x) = G(x,\cdot).
\]
Then the preimage of every open ball in $\mathcal F(\Y)$ under $\widetilde G$ is
measurable in $\mathbb X$. If the range $\widetilde G(\mathbb X)$ is separable
under $\widetilde d$, then $\widetilde G$ is
$\mathcal A/\mathcal B(\mathcal F(\Y))$-measurable and is strongly measurable
with respect to every measure on $(\mathbb X,\mathcal A)$.
\end{theorem}

\begin{proof}
Fix $f\in\mathcal F(\Y)$. By assumption,
\[
\widetilde d\bigl(\widetilde G(x),f\bigr)
=
\rho\bigl(\Phi(\widetilde G(x)),\Phi(f)\bigr).
\]
Since
\[
\Phi(\widetilde G(x)) = \bigl(G(x,y_n)\bigr)_{n\ge1},
\]
and each map $x\mapsto G(x,y_n)$ is $\mathcal A$-measurable, it follows that
\[
x\mapsto \Phi(\widetilde G(x))
\]
is $\mathcal A/\mathcal B(\mathbb R^{\mathbb N})$-measurable. As $\rho$ is
Borel measurable and $\Phi(f)$ is fixed, the map
\[
x\mapsto \widetilde d\bigl(\widetilde G(x),f\bigr)
\]
is $\mathcal A$-measurable.
 
Hence, for every $f\in\mathcal F(\Y)$ and $r>0$,
\[
\widetilde G^{-1}\bigl(B_{\widetilde d}(f,r)\bigr)
=
\{x\in\mathbb X : \widetilde d(\widetilde G(x),f)<r\}
\in \mathcal A.
\]
Suppose now that $R:=\widetilde G(\mathbb X)$ is separable, and choose a countable
dense subset $(f_k)_{k\geq1}$ of $R$. For every open set
$O\subseteq\mathcal F(\Y)$, the relative open set $O\cap R$ is the union of a
countable collection of relative balls $R\cap B(f_k,q)$, where $q>0$ is rational
and the closure of the ambient ball $B(f_k,q)$ is contained in $O$. Since $\widetilde G$ takes values in
$R$, the inverse image $\widetilde G^{-1}(O)$ is therefore a countable union of
measurable inverse images of balls. Thus $\widetilde G$ is Borel measurable.
Its range is separable, so Theorem~\ref{thm:pettis-metric} shows that it is
strongly measurable with respect to every measure on
$(\mathbb X,\mathcal A)$.
\end{proof}

\begin{corollary}[Essential-range verification]
\label{cor:essential-range-verification}
In the setting of Theorem~\ref{thm:measurable-function-valued}, let $\mu$ be a
measure on $(\mathbb X,\mathcal A)$. If $\widetilde G$ has an essentially
separable range, then it agrees $\mu$-almost everywhere with a measurable,
strongly measurable map. In particular, $\widetilde G$ is strongly measurable
in the usual almost-everywhere sense. If $(\mathbb X,\mathcal A,\mu)$ is
complete, then $\widetilde G$ itself is also Borel measurable.
\end{corollary}

\begin{proof}
Choose $N\in\mathcal A$ with $\mu(N)=0$ such that
$\widetilde G(\mathbb X\setminus N)$ is separable. If $\mathbb X$ is empty,
the conclusion is immediate. Otherwise, fix
$f_0\in\mathcal F(\Y)$ and define $H=\widetilde G$ on
$\mathbb X\setminus N$ and $H=f_0$ on $N$. The scalar maps
$x\mapsto H(x)(y)$ remain measurable for every $y\in\mathbb Y$, and the range
of $H$ is separable. Theorem~\ref{thm:measurable-function-valued} therefore
implies that $H$ is measurable and strongly measurable. Since
$H=\widetilde G$ almost everywhere, the remaining assertions follow.
\end{proof}

\begin{corollary}[Lipschitz-valued range verification]
\label{cor:lipschitz-range-verification}
Let $(\mathbb Y,d)$ be a nonempty metric space, let
$\mathcal F(\Y)\subseteq\operatorname{Lip}_b(\Y)$, and equip
$\mathcal F(\Y)$ with
\[
\|f\|_{\mathrm{Lip}}
:=
\sup_{y\in\Y}|f(y)|
+
\sup_{y\neq z}\frac{|f(y)-f(z)|}{d(y,z)}.
\]
Here the second supremum is understood to be zero if $\Y$ has only one point.
Let $(\mathbb X,\mathcal A,\mu)$ be a measure space and suppose that
$G\colon\mathbb X\times\mathbb Y\to\R$ satisfies the two pointwise assumptions
of Theorem~\ref{thm:measurable-function-valued}. If the associated map
$\widetilde G$ has essentially separable range in
$\|\cdot\|_{\mathrm{Lip}}$, then $\widetilde G$ is strongly measurable.
\end{corollary}

\begin{proof}
Outside a null set, let the range of $\widetilde G$ be contained in a separable
set $R$, and replace $R$ by its closure. If $R$ is empty, the conclusion is
immediate. Otherwise, the difference set
$R-R:=\{f-g:f,g\in R\}$ is separable in $\|\cdot\|_{\mathrm{Lip}}$. Choose a
countable dense subset $(h_j)_{j\geq1}$ of $R-R$. For every $j,m\geq1$, choose
points that approximate each nonempty supremum defining
$\|h_j\|_{\mathrm{Lip}}$ within $1/m$, and let $D\subseteq\Y$ be the union of
all points so chosen. If necessary, add one point of $\Y$ to $D$. Then $D$ is
countable and, by density of the $h_j$,
\[
\|h\|_{\mathrm{Lip}}
=
\sup_{y\in D}|h(y)|
+
\sup_{\substack{y,z\in D\\y\neq z}}
\frac{|h(y)-h(z)|}{d(y,z)}
\qquad (h\in R-R).
\]
After enumerating $D$, the right-hand side is a Borel function of the two
evaluation sequences. Thus the metric on $R$ is reconstructed from countably
many evaluations. Modifying $\widetilde G$ on the null set to take one fixed
value in $R$, Theorem~\ref{thm:measurable-function-valued}, with $R$ in place
of $\mathcal F(\Y)$, yields a strongly measurable modification. Hence
$\widetilde G$ is strongly measurable.
\end{proof}

Example~\ref{ex:not-measurable} shows why the separability step cannot simply be
omitted: measurable inverse images of all balls need not imply Borel
measurability in a non-separable target. The results above isolate the useful
replacement in applications: countably many measurable evaluations, together
with separability of the range, yield measurability and strong measurability.

%%%%%%%%%%%%%%%%%%%%%%%%%%%%%%%%%%%%%%%%%%%%%%%%%%%%%%%%%%%%%%
%%%%%%%%%%%%%%%%%%%%%%%%%%%%%%%%%%%%%%%%%%%%%%%%%%%%%%%%%%%%%%

\section{Adapted Approximations}\label{sect:applications}
We now address the second main objective of this paper: adapted approximation.
Using the strong-measurability criteria and the approximation lemmas established
above, we construct elementary approximations of metric-space-valued stochastic
processes, obtain approximation in the full Lipschitz norm for function-valued
processes under the stated function-space hypotheses, and develop smooth
approximations based on finitely many past observations.

\subsection[Elementary adapted approximations]{{Elementary adapted approximations}}

We first recall some standard definitions and facts
about stochastic processes and the associated $\sigma$-algebras, and then we
present the approximation results.
 
Let $(\Omega,\F,(\F_t)_{t \in [0,T]},\mathbb P)$ be a filtered probability space. We say that $(\F_t)_{t \in [0,T]}$ satisfies the \emph{usual
  conditions} if it is right-continuous and $\F_0$ contains all $\mathbb P$-null sets of $\F$.
We recall a few standard $\sigma$-algebras on $[0,T] \times \Omega$ (see \cite{KaratzasShreve}).

\begin{definition}
  The following $\sigma$-algebras on $[0,T] \times \Omega$ are defined:
  \begin{itemize}
    \item The \emph{adapted} $\sigma$-algebra $\mathcal{A}$ is
          \[
            \mathcal{A} :=
            \bigl\{
            A \in \B([0,T]) \otimes \F :
            A_t := \{\omega \in \Omega : (t,\omega) \in A\} \in \F_t \;\forall t \in [0,T]
            \bigr\}.
          \]
    \item The \emph{progressive} $\sigma$-algebra $\mathrm{Prog}$ is
          \[
            \mathrm{Prog} :=
            \bigl\{
            A \in \B([0,T]) \otimes \F :
            A \cap ([0,t] \times \Omega) \in \B([0,t]) \otimes \F_t
            \;\forall t \in [0,T]
            \bigr\}.
          \]
    \item The \emph{predictable} $\sigma$-algebra $\mathcal{P}$ is the smallest
          $\sigma$-algebra on $[0,T] \times \Omega$ with respect to which all
          left-continuous adapted processes are measurable.
  \end{itemize}
\end{definition}
\noindent
A process $X\colon [0,T] \times \Omega \to \R$ is said to be
\begin{itemize}
  \item measurable if it is $\B([0,T]) \otimes \F$-measurable.
  \item jointly measurable and adapted if it is $\mathcal{A}$-measurable. In other words,
        it is measurable, and $X_t$ is $\F_t$-measurable for each $t$.
  \item progressive if it is measurable with respect to $\mathrm{Prog}$; and
  \item predictable if it is $\mathcal{P}$-measurable.
\end{itemize}

We begin by recalling the notion of an elementary process.
\begin{definition}[Elementary process]
  A real-valued process $H\colon [0,T] \times \Omega \to \R$ is called an
  \emph{elementary predictable process} if there exists a partition
  $0 = t_0 < t_1 < \dots < t_n = T$ of $[0,T]$ and bounded random variables
  $\xi_0,\xi_1,\dots,\xi_{n-1}$ such that:
  \begin{itemize}
    \item for each $i=0,\dots,n-1$, $\xi_i$ is $\F_{t_i}$-measurable; and
    \item the process $H$ is given by
          \begin{equation}\label{eq:elementary}
            H(\omega,t) = \xi_0(\omega)\,\1_{\{0\}}(t)
            + \sum_{i=0}^{n-1} \xi_i(\omega)\,\1_{(t_i,t_{i+1}]}(t).
          \end{equation}
  \end{itemize}
  In the sequel we simply refer to such processes as
  \emph{elementary processes}.
\end{definition}

In Karatzas and Shreve~\cite[Lemma~3.2.4]{KaratzasShreve}, it is proved that if the
filtration $(\F_t)_{t \in [0,T]}$ satisfies the usual conditions, then every bounded,
jointly measurable adapted process can be approximated in $L^2([0,T] \times \Omega)$ by elementary
predictable processes. A careful reading of the proof shows that:

\begin{itemize}
  \item the usual conditions on $(\F_t)_{t \in [0,T]}$ are not needed if one assumes the
        process is progressively measurable rather than jointly measurable and adapted; and
  \item  {the approximation extends to $L^p$ for every $1\le p<\infty$.}
\end{itemize}

We summarise this in the following theorem with proof omitted.
\begin{theorem}\label{thm:elem-approx}
  Let $(\Omega,\F,(\F_t)_{t \in [0,T]},\mathbb P)$ be a filtered probability space, and
  let \textcolor{blue}{$1\le p<\infty$}. Then:
  \begin{enumerate}[(a)]
    \item For every progressively measurable process
          $X \in L^p(\Omega \times [0,T],\,\mathrm{Prog},\mathbb P \otimes \lambda)$ and
          every $\varepsilon>0$, there exists an elementary process $H$
          such that
          \[
            \E \Bigl[ \int_0^T |X(\omega,t) - H(\omega,t)|^p
              \,dt \Bigr] \le \varepsilon.
          \]
    \item If, in addition, $(\F_t)_{t \in [0,T]}$ satisfies the usual conditions, then the same
          conclusion holds for every adapted process
          $X \in L^p(\Omega \times [0,T],\,\mathcal{A},\mathbb P \otimes \lambda)$.
  \end{enumerate}
\end{theorem}

% {\color{blue}
% \begin{proof}
% First truncate $X$; the truncations converge in $L^p(\mathbb P\otimes\lambda)$.
% For a bounded progressive process, the backward averages
% \[
%   X_t^\delta:=\frac1\delta\int_{(t-\delta)^+}^t X_s\,ds
% \]
% are bounded continuous adapted processes and converge to $X$ in this $L^p$
% space by Lebesgue differentiation and bounded convergence. For each fixed
% $\delta$, their left-endpoint discretizations are elementary processes and
% converge in $L^p$ by continuity and bounded convergence. This proves (a).
% For (b), apply the cited lemma of Karatzas and Shreve to bounded $X$, and clip
% the elementary $L^2$ approximants to the same bound $M>0$ as $X$. Their $L^p$
% convergence follows from H\"older's inequality when $1\le p<2$, and from
% $|X-H|^p\le(2M)^{p-2}|X-H|^2$ when $p\ge2$. Finally, remove the truncation.
% \end{proof}
% }

\begin{remark}\label{rem:elementary-indicator}
  For every $E$ in one of the $\sigma$-algebras
  $\mathcal{M}=\mathrm{Prog}$ or $\mathcal{M}=\mathcal{A}$ (according to whether
  the usual conditions hold), and every $\varepsilon>0$, there exists an
  elementary process $H$ of the form \eqref{eq:elementary} such that
  \[
    \E \Bigl[ \int_0^T |\1_E(\omega,t) - H(\omega,t)|^p \,dt \Bigr]
    \le \varepsilon/2^p.
  \]
  However, for our purposes
  we need the coefficients $\xi_i$, in its expression like \eqref{eq:elementary}, to be indicator functions of events in
  $\F_{t_i}$, rather than arbitrary $\F_{t_i}$-measurable random variables. For this,
  we can threshold the $\xi_i$'s to get processes of the form
  \[
    \widetilde{H}(t,\omega) = \1_{\{0\}}(t)\,\1_{A_0}(\omega) + \sum_{i=0}^{n-1} \1_{(t_i,t_{i+1}]}(t)\,\1_{A_i}(\omega),
  \]
  where $A_i = \{\omega \in \Omega : \xi_i(\omega) \geq 1/2\}\in \mathcal{F}_{t_i}$.
  Now using the fact that
  \[
    |\1_E(\omega,t) - \widetilde{H}(\omega,t)| \le 2|\1_E(\omega,t) - H(\omega,t)|,
  \]
  we see that
  \[
    \E \Bigl[ \int_0^T |\1_E(\omega,t) - \widetilde{H}(\omega,t)|^p \,dt \Bigr]
    \le \varepsilon.
  \]
\end{remark}

Combining the remark above with Lemma~\ref{lem:approx} and the strong
measurability results, we obtain an $L^p$-approximation theorem for
metric-space-valued stochastic processes.

\begin{theorem}\label{thm:banach-elem-approx}
  Let $(\Omega,\F,(\F_t)_{t \in [0,T]},\mathbb P)$ be a filtered probability
  space, let $(\Y, d)$ be a metric space, and fix $1\le p<\infty$.
  Let $\mathcal{M}$ equal $\mathrm{Prog}$ when
  $(\F_t)_{t \in [0,T]}$ does not necessarily satisfy the usual conditions and $\mathcal{A}$
  when it does. Assume that $\operatorname{dens}(\Y)$ is strictly smaller than
  every real-valued measurable cardinal. Then for every
   {$(\mathcal M,\B_\Y)$-measurable} stochastic process
  \[
    X\colon (\Omega \times [0,T],\mathcal{M}) \to (\Y,\B_\Y)
  \]
  satisfying
  \[
    \E \left[ \int_0^T d(X(\omega,t), y_0)^p \,dt \right] < \infty
    \]
  for some $y_0 \in \Y$,
  and every $\varepsilon>0$, there exists a process $\widetilde{X}$ of the form
  \[
    \widetilde{X}(\omega,t)
    = \1_{\{0\}}(t)\,\Xi_0(\omega) + \sum_{i=0}^{n-1} \1_{(t_i,t_{i+1}]}(t)\,\Xi_i(\omega),
  \]
  where $0 = t_0 < t_1 < \dots < t_{n} = T$ is a partition of $[0,T]$ and, for each
  $i$, we have
  \[
    \Xi_i(\omega) = \sum_{j=1}^{N} y_{i,j}\,\1_{A_{i,j}}(\omega)
  \]
  for some $y_{i,1},\dots,y_{i,N} \in \Y$ and a partition
  $A_{i,1},\dots,A_{i,N} \in \F_{t_i}$ of $\Omega$  {(allowing empty cells)},
  such that
  \[
    \E \Bigl[ \int_0^T d(X(\omega,t), \widetilde{X}(\omega,t))^p \,dt \Bigr]^{1/p}
    \le \varepsilon.
  \]
\end{theorem}

\begin{proof}
  By Lemma~\ref{lem:approx} and Theorem~\ref{thm:second}(i), every measurable
  metric-space-valued process as above is strongly measurable
  and can be approximated by simple processes of the form
  \begin{equation}\label{eq:simple-process}
    S(\omega,t)
    = \sum_{k=1}^N y_k \,\1_{E_k}(\omega,t),
  \end{equation}
  with $E_1,\dots,E_N \in \mathcal{M}$ forming a partition of $\Omega\times [0,T]$, and $y_k \in \Y$ in the sense that
  \[
    \E \Bigl[ \int_0^T d(X(\omega,t), S(\omega,t))^p \,dt \Bigr]^{1/p} \le \varepsilon/2.
  \]
  Fix $D > \max_{1 \le k, l \le N} d(y_k, y_l)$.
  For each indicator process $\1_{E_k}$, we apply Remark~\ref{rem:elementary-indicator}
  to approximate it in $L^p(\Omega \times [0,T], \mathcal{M}, \mathbb P \otimes \lambda)$
  up to an $L^p$-norm error of $\varepsilon/(2ND)$ by an elementary indicator process
  \[
    \widetilde{H}_k =
    \1_{\{0\}}(t)\,\1_{A_{k,0}}(\omega) + \sum_{i=0}^{n-1} \1_{(t_i,t_{i+1}]}(t)\,\1_{A_{k,i}}(\omega),
  \]
  where the partition $0 = t_0 < t_1 < \dots < t_{n} = T$ is the same for all $k$.
  For each $i$, turn the sets $A_{k,i}$ into a measurable partition by setting
  \[ 
    B_{k,i}:=A_{k,i}\setminus\bigcup_{\ell<k}A_{\ell,i}\quad (k<N),
    \qquad B_{N,i}:=\Omega\setminus\bigcup_{k<N}A_{k,i}.
  \]
  Now define \[
    \widetilde{X}(\omega,t) := y_k
  \]
  whenever $\omega \in B_{k,i}$ where $i$ is such that $t \in (t_i,t_{i+1}]$
  or $i=0$ if $t=0$.
  If all the indicators $\widetilde H_k(\omega,t)$ agree with
  $\1_{E_k}(\omega,t)$, the selected value equals $S(\omega,t)$; otherwise the
  distance is at most $D$. Consequently,
  \[
    d(S(\omega,t),\widetilde X(\omega,t))
    \le D\sum_{k=1}^N|\1_{E_k}(\omega,t)-\widetilde H_k(\omega,t)|,
  \]
  so its $L^p$ norm is at most $\varepsilon/2$.
  Thus $\widetilde{X}$ is of the desired form and
  \begin{align*}
    \E \Bigl[ \int_0^T d(X(\omega,t), \widetilde{X}(\omega,t))^p \,dt \Bigr]^{1/p}
     & \le  \E \Bigl[ \int_0^T d(X(\omega,t), S(\omega,t))^p \,dt \Bigr]^{1/p}                \\
     & \quad + \E \Bigl[ \int_0^T d(S(\omega,t), \widetilde{X}(\omega,t))^p \,dt \Bigr]^{1/p} \\
     & \le \frac{\varepsilon}{2} + \frac{\varepsilon}{2} = \varepsilon.
  \end{align*}
\end{proof}

\subsection[Lipschitz-valued adapted approximations]{ {Lipschitz-valued adapted approximations}}
A particularly important class of metric spaces for our purposes is provided by
Wasserstein spaces. These spaces arise naturally in mean field control problems
and mean field games, where probability measures appear as state variables; see \cite{bensoussan2013mean,cardaliaguet2019master,carmona2018probabilistic}. We
therefore recall the definition of Wasserstein spaces and the associated
Wasserstein distances. For more details, see, for example, Villani~\cite{Villani2009}
or Ambrosio, Gigli, and Savar\'e~\cite{AmbrosioGigliSavare2008}.

\begin{definition}[Wasserstein spaces and Wasserstein distances]
  Let $(\mathbb S,\rho)$ be a metric space, let $q\in[1,\infty)$, and let
  $\mathcal P(\mathbb S)$ denote the set of Radon probability measures on
  $\mathbb S$, that is, Borel probability measures that are inner regular
  by compact sets. On Polish spaces, every Borel probability measure is
  Radon. The \emph{$q$-Wasserstein space} over $\mathbb S$ is defined by
  \[
    \mathcal P_q(\mathbb S)
    :=
    \left\{
      \mu\in\mathcal P(\mathbb S):
      \int_{\mathbb S}\rho(x,x_0)^q\,\mu(dx)<\infty
      \text{ for some, equivalently every, } x_0\in\mathbb S
    \right\}.
  \]
  For $\mu,\nu\in\mathcal P_q(\mathbb S)$, the \emph{$q$-Wasserstein distance}
  between $\mu$ and $\nu$ is given by
  \[
    W_q(\mu,\nu)
    :=
    \left(
      \inf_{\pi\in\Pi(\mu,\nu)}
      \int_{\mathbb S\times\mathbb S}\rho(x,y)^q\,\pi(dx,dy)
    \right)^{1/q},
  \]
  where $\Pi(\mu,\nu)$ denotes the set of all couplings of $\mu$ and $\nu$; that
  is, the set of Radon probability measures on $\mathbb S\times\mathbb S$ whose first
  and second marginals are $\mu$ and $\nu$, respectively.
\end{definition}

The function-valued consequence of Theorem~\ref{thm:banach-elem-approx} can
be stated independently of the number of parameter variables.

\begin{corollary}[Lipschitz-valued adapted approximation]
\label{cor:lip-field-approx}
Let $(\Omega,\F,(\F_t)_{t\in[0,T]},\mathbb P)$ be a filtered probability
space, with $\mathcal M=\mathcal A$ under the usual conditions and
$\mathcal M=\mathrm{Prog}$ otherwise. Let $(\Y,d)$ be a nonempty metric space
and let $\mathcal F_L(\Y)$ be the bounded real-valued Lipschitz functions on
$\Y$ with Lipschitz constant at most $L>0$, equipped with the full Lipschitz
norm of Corollary~\ref{cor:lipschitz-range-verification}.
Suppose $G\colon\Omega\times[0,T]\times\Y\to\R$ is uniformly bounded,
$G(\cdot,\cdot,y)$ is $\mathcal M$-measurable for every $y$, and its sections
belong to $\mathcal F_L(\Y)$ outside a common $\mathcal M$-measurable
$\mathbb P\otimes\lambda$-null set. Set $G=0$ on that set and assume that
$\widetilde G(\omega,t)=G(\omega,t,\cdot)$ satisfies either:
\begin{enumerate}[(a)]
  \item $\widetilde G$ is $\mathcal M/\mathcal B_{\mathcal F_L(\Y)}$-measurable
        and $\operatorname{dens}(\mathcal F_L(\Y))$ is strictly smaller than
        every real-valued measurable cardinal;
  \item $\widetilde G$ has essentially separable range in
        $\|\cdot\|_{\mathrm{Lip}}$.
\end{enumerate}
Then, for every $1\le p<\infty$ and $\varepsilon>0$, there is an elementary
function-valued process $\bar G$ of the form in
Theorem~\ref{thm:banach-elem-approx}, with deterministic values in
$\mathcal F_L(\Y)$, such that
\[
  \left(\E\int_0^T
  \|\bar G(\omega,t,\cdot)-G(\omega,t,\cdot)\|_{\mathrm{Lip}}^p\,dt
  \right)^{1/p}\le\varepsilon.
\]
If, in addition, $Q_t\colon\Y\to\Y$ are nonexpansive projections with
$t\mapsto Q_ty$ continuous for every $y$, and
$G(\omega,t,\cdot)\circ Q_t=G(\omega,t,\cdot)$ outside a common
$\mathcal M$-measurable null set, then
\[
  \widehat G(\omega,t,\cdot):=\bar G(\omega,t,\cdot)\circ Q_t
\]
is invariant under $Q_t$, is progressively measurable for each fixed $y$,
and satisfies the same $L^p$ error bound in the full Lipschitz norm.
The stopped error norm is measurable for the completed product measure;
$\widehat G$ need not be elementary as a function-valued process.
\end{corollary}

\begin{proof}
Theorem~\ref{thm:second} under (a), or
Corollary~\ref{cor:lipschitz-range-verification} under (b), gives a measurable,
strongly measurable version of $\widetilde G$.
Apply Theorem~\ref{thm:banach-elem-approx} directly under (a), and to this
version with its separable range (together with zero) as target under (b).
Uniform boundedness and the bound $L$ ensure integrability for every finite $p$.
For the stopped process, each deterministic coefficient $h(Q_ty)$ is
continuous in $t$, and the elementary indicator processes are predictable;
hence $\widehat G(\cdot,\cdot,y)$ is progressive. Since $Q_t^2=Q_t$,
$\widehat G(\omega,t,\cdot)\circ Q_t=\widehat G(\omega,t,\cdot)$.
Nonexpansivity and the invariance of $G$ give, almost everywhere,
\[
  \|\widehat G(\omega,t,\cdot)-G(\omega,t,\cdot)\|_{\mathrm{Lip}}
  \le \|\bar G(\omega,t,\cdot)-G(\omega,t,\cdot)\|_{\mathrm{Lip}}.
\]
For fixed $h\in\operatorname{Lip}_b(\Y)$,
$t\mapsto\|h\circ Q_t\|_{\mathrm{Lip}}$ is lower semicontinuous: both
suprema defining the norm are over continuous functions of $t$.
Approximate the strongly measurable difference $\bar G-G$ by simple
functions and use
$\|(h-h')\circ Q_t\|_{\mathrm{Lip}}\le\|h-h'\|_{\mathrm{Lip}}$.
The stopped error norm is therefore measurable for the completed product
measure, so the displayed inequality can be integrated.
\end{proof}

% {\color{blue}
% Approximation in the full Lipschitz norm requires more than pointwise
% measurability and uniform Lipschitz bounds. For example, on
% $[0,1]\times[0,1]$, the deterministic field $G(t,y)=|y-t|$ is bounded,
% jointly continuous, and uniformly $1$-Lipschitz in $y$, but
% \[
%   \operatorname{Lip}\bigl(G(t,\cdot)-G(s,\cdot)\bigr)=2
%   \qquad(t\ne s).
% \]
% Consequently, its lifted range is not essentially separable in the full
% Lipschitz norm with respect to Lebesgue measure in time. Thus pointwise
% measurability and uniform Lipschitz bounds alone do not imply the
% function-space hypotheses of Corollary~\ref{cor:lip-field-approx}.
% }

For the path and Wasserstein setting, take a nonempty metric space
$(\mathbb S,\rho)$, $1\le q<\infty$, and
\[
  \Y=\R^d\times C([0,T];\R^d)\times\mathcal P_q(\R^d)
      \times\mathbb S\times\mathcal P_q(\mathbb S),
\]
with the sum of the Euclidean, supremum, $W_q$, $\rho$, and $W_q$ metrics.
Let $Q_t$ replace the path $\phi$ by $\phi_{\cdot\wedge t}$ and leave all
other coordinates unchanged. These are nonexpansive projections with
continuous orbits. A Lipschitz condition using only the path up to time $t$ reads
\[
  |G(\omega,t,y)-G(\omega,t,y')|\le L\,d(Q_ty,Q_ty')
  \qquad(y,y'\in\Y)
\]
outside a common $\mathcal M$-measurable null set. It implies both the
full-metric Lipschitz bound and invariance under $Q_t$.
For uniformly bounded, pointwise $\mathcal M$-measurable $G$ satisfying either
alternative above, the corollary thus gives
non-anticipative approximations controlling both the supremum norm and the
Lipschitz seminorm of the error.

Here the parameter paths are continuous, which ensures continuity of
$t\mapsto Q_t\phi$ in the supremum metric. For c\`adl\`ag parameter paths,
this continuity can fail at jump times, so the stopped-field conclusion
requires a separate argument. The next subsection allows jumps in the
observation process generating the filtration.

%%%%%%%%%%%%%%%%%%%%%%%%%%%%%%%%%%%%%%%%%
%%%%%%%%%%%%%%%%%%%%%%%%%%%%%%%%%%%%%%%%%

\subsection[Smooth approximations from finitely many past observations]{{Smooth approximations from finitely many past observations}}
We now refine the approximation result of Theorem~\ref{thm:banach-elem-approx}
in the case where the filtration is generated by an underlying stochastic
process.  More precisely, we assume that the information flow is generated by a
real-valued adapted process $U$, so that, for every $t\in[0,T]$,
\[\mathcal F_t=\sigma(U_s:0\le s\le t),\]
or every event in $\mathcal F_t$ agrees almost surely with an event in this
natural filtration. We assume that $U$ has left- or right-continuous paths.
The results below give progressive approximations whose coefficients are smooth
functions of finitely many observations of $U$, including when $U$ has jumps.
They also permit smoothing in time.
Section~\ref{sect:nonmarkovian-control} develops a control application and
discusses related literature.
 
The following lemma concerns approximation of random variables on a probability
space whose $\sigma$-algebra is generated by a countable family of random
variables.

\begin{lemma}\label{lem:smooth-Xn}
  Let $(\Omega,\F,\mathbb P)$ be a probability space such that every event in
  $\F$ agrees almost surely with an event in $\sigma(X_n:n\ge1)$, for a
  countable family of real-valued random variables $\{X_n\}_{n=1}^\infty$.
  Then for every $1\le p<\infty$, every real-valued
  $\xi \in L^p(\Omega,\F,\mathbb P)$, and every $\varepsilon>0$, there
  exist indices $1 \le i_1 < i_2 < \dots < i_m$ and a function
  $f \in C_c^\infty(\R^m)$ such that
  \[
    \bigl( \E\bigl[ |\xi - f(X_{i_1},\dots,X_{i_m})|^p \bigr] \bigr)^{1/p}
    \le \varepsilon.
  \]
  If $0\le\xi\le1$ almost surely, $f$ may also be chosen to satisfy
  $0\le f\le1$.
\end{lemma}

\begin{proof}
  First, we may assume without loss of generality that $\xi$ is bounded by some
  constant $C>0$, since bounded functions are dense in $L^p(\Omega,\F,\mathbb P)$.
\\[4pt]
  By assumption, $\xi$ has a $\sigma(X_n:n\ge1)$-measurable version.
  Set $\mathcal{H}_n := \sigma(X_1,\dots,X_n)$ for each $n \in \N$. By the martingale
  convergence theorem, the sequence
  \[
    M_n := \E[\xi \mid \mathcal{H}_n]
  \]
  converges to $\xi$ almost surely as $n \to \infty$. Since the $M_n$ are bounded
  (by $C$), the bounded convergence theorem implies that $M_n \to \xi$ in
  $L^p(\Omega,\F,\mathbb P)$ as $n \to \infty$. Thus we can choose $n$ such that
  \[
    \E[|\xi - M_n|^p]^{1/p} \le (\varepsilon/3).
  \]
  Because $M_n$ is $\mathcal{H}_n$-measurable, there exists a measurable function
  $g\colon (\R^n,\B_{\R^n}) \to (\R,\B_{\R})$ such that
  \[
    M_n = g(X_1,\dots,X_n).
  \]
  Without loss of generality, we may take $|g| \le C$; otherwise we use $(-C)\vee (g \wedge C)$ instead.
  % This is because
  % we can replace $g$ by its truncation at $\pm C$.
  Let $\mathbb P_X$ denote the law of $X := (X_1,\dots,X_n)$ on $\R^n$. By
  \cite[Theorem~7.1.7]{BogachevII}, $\mathbb P_X$ is a regular Borel probability measure
  on $\R^n$. By Lusin's theorem (see, for example, \cite[Theorem~2.24]{Rudin}), for
  every $\delta>0$ there exists a continuous function $\widetilde{g}\colon \R^n
    \to \R$ with compact support such that $|\widetilde{g}| \le C$ and
  \[
    \mathbb P(g(X) \ne \widetilde{g}(X)) < \delta.
  \]
  Choose $\delta>0$ so small that
  \[
    \delta^{1/p} (2C) \le \varepsilon/3,
  \]
  and let $\widetilde{g}$ be as above. Then
  \[
    \left(\int_{\R^n} |g - \widetilde{g}|^p \,d\mathbb P_X\right)^{1/p}
    = \left(\int_{\{g \ne \widetilde{g}\}} |g - \widetilde{g}|^p \,d\mathbb P_X\right)^{1/p}
    \le (2C) \delta^{1/p} \le (\varepsilon/3).
  \]
\\[4pt]
  Next, since $\widetilde{g}$ is continuous with compact support, it is uniformly
  continuous. Hence there exists $\theta>0$ such that
  \[
    \|\mathbf{x} - \mathbf{y}\|_2 \le \theta
    \quad\Longrightarrow\quad
    |\widetilde{g}(\mathbf{x}) - \widetilde{g}(\mathbf{y})| \le \varepsilon/3
  \]
  for all $\mathbf{x},\mathbf{y} \in \R^n$. Let $\eta$ be a standard mollifier
  supported in the Euclidean ball of radius $\theta$ around the origin. Define
  \[
    f(\mathbf{x}) := (\eta * \widetilde{g})(\mathbf{x})
    := \int_{\R^n} \eta(\mathbf{y} - \mathbf{x}) \,\widetilde{g}(\mathbf{y})\,d\mathbf{y}.
  \]
  Then $f \in C_c^\infty(\R^n)$ and, for every $\mathbf{x} \in \R^n$,
  \[
    |f(\mathbf{x}) - \widetilde{g}(\mathbf{x})|
    = \Bigl| \int_{\R^n} \eta(\mathbf{y} - \mathbf{x})
    \bigl( \widetilde{g}(\mathbf{y}) - \widetilde{g}(\mathbf{x}) \bigr)
    \,d\mathbf{y} \Bigr|
    \le \varepsilon/3,
  \]
  since the integrand is non-zero only when $\|\mathbf{y} - \mathbf{x}\|_2 \le
    \theta$.
 
  Consequently,
  \[
    \left(\int_{\R^n} |f - \widetilde{g}|^p \,d\mathbb P_X\right)^{1/p} \le (\varepsilon/3).
  \]
  Combining the estimates and using the triangle inequality, we obtain
  \[
    \E\bigl[|\xi - f(X_1,\dots,X_n)|^p\bigr]^{1/p}
    \le \E[|\xi - M_n|^p]^{1/p}
    + \E[|M_n - \widetilde{g}(X)|^p]^{1/p}
    + \E[|\widetilde{g}(X) - f(X)|^p]^{1/p}
    \le \varepsilon.
  \]
  This proves the claim (with $m=n$ and $i_j=j$). If $0\le\xi\le1$ almost
  surely, choose $0\le g\le1$ and replace the continuous compactly supported
  extension $\widetilde g$ by $0\vee(\widetilde g\wedge1)$. Convolution with
  the nonnegative mollifier then gives $0\le f\le1$, with the same argument.
\end{proof}

\begin{remark}
  Lemma~\ref{lem:smooth-Xn} requires no distributional assumptions on the
  generating random variables, such as absolute continuity with respect to
  Lebesgue measure. In particular, let $U$ be a real-valued adapted process with
  left- or right-continuous paths. On every $[0,t]\subseteq[0,T]$, its
  observations at rational times, together with $U_0$ and $U_t$, generate
  $\sigma(U_s:0\le s\le t)$ by one-sided continuity. If every event in
  $\F_t$ agrees almost surely with an event in this observation
  $\sigma$-algebra, the lemma therefore approximates every real-valued
  $\xi\in L^p(\Omega,\F_t,\mathbb P)$, $1\le p<\infty$, to arbitrary
  $L^p$-accuracy by smooth compactly supported functions of finitely many
  observations up to $t$. These functions can be chosen between $0$ and $1$
  when $0\le\xi\le1$ almost surely. For $t=0$, only $U_0$ is needed.
\end{remark}
 
By combining Lemma~\ref{lem:smooth-Xn} and Theorem~\ref{thm:banach-elem-approx},
we can obtain a smooth approximation theorem for stochastic processes with values
in a metric space. Since the construction involves linear combinations of the
process values, we now restrict the state space to a topological vector space
whose topology is induced by a metric. The relevant background on topological
vector spaces can be found, for example, in \cite{SchaeferWolff1999}.

\begin{theorem}\label{thm:smooth-process-approx}
  Let $(\Omega,\F,(\F_t)_{0\le t\le T},\mathbb P)$ be a filtered probability
  space and let $U$ be a real-valued stochastic process with left- or
  right-continuous paths. Assume that, for every $t\in[0,T]$,
  \[
    \mathcal F_t^U:=\sigma(U_s:0\le s\le t)\subseteq\mathcal F_t,
  \]
  and every event in $\mathcal F_t$ agrees almost surely with an event in
  $\mathcal F_t^U$. In particular, this holds for the natural filtration of
  $U$ and for its completion by null events of $\F$. Let $1\le p<\infty$ and let
   {$\Y$ be a metrizable topological vector space equipped with a
  compatible metric $d$}, whose density is strictly smaller than every
  real-valued measurable cardinal. Then for every
   {$(\mathcal M,\B_\Y)$-measurable} stochastic process
  \[
    X\colon (\Omega \times [0,T],\mathcal{M}) \to (\Y,\B_\Y),
  \]
  satisfying
  \[
    \E\Bigl[ \int_0^T d(X(\omega,t),y_0)^p \,dt \Bigr] < \infty
  \]
  for some $y_0 \in \Y$,
  where $\mathcal{M}=\mathrm{Prog}$, or $\mathcal{M}=\mathcal{A}$ if the
  filtration also satisfies the usual conditions,
  and every $\varepsilon>0$, there exists a process
  $\widetilde{X}$ of the form
  \begin{equation}\label{eq:smooth-process}
    \widetilde{X}(\omega,t)
    = \sum_{j=1}^{N}\1_{\{0\}}(t)\,f_{0,j}(U_0(\omega))\,y_{0,j} +
    \sum_{j=1}^{N}\sum_{i=0}^{n-1} \1_{(t_i,t_{i+1}]}(t)\,
    f_{i,j}(U_{t_0}(\omega),\dots,U_{t_i}(\omega))\,y_{i,j},
  \end{equation}
  where $0 = t_0 < t_1 < \dots < t_{n} = T$ is a partition of $[0,T]$,
  $f_{i,j} \in C_c^\infty(\R^{i+1})$ with $0\le f_{i,j}\le1$ and
  $y_{i,j} \in \Y$ for each $i$ and $j$, such that
  \[
    \E \Bigl[ \int_0^T d(X(\omega,t), \widetilde{X}(\omega,t))^p \,dt \Bigr]^{1/p}
    \le \varepsilon.
  \]
  In addition, with a suitable choice of these coefficients, the same error
  bound holds after omitting the term supported at $\{0\}$ and replacing
  $\1_{(t_i,t_{i+1}]}$, for $i=0,\dots,n-1$, by functions
  $\varphi_i\in C_c^\infty(\R)$ with $0\le\varphi_i\le1$ and
  $\operatorname{supp}(\varphi_i)\subset(t_i,t_{i+1})$ that approximate these
  indicators in $L^p([0,T])$. The resulting process is still progressive.

\end{theorem}

% {\color{blue}
% The generation hypothesis on $\mathcal F_t$ is preserved by completion with
% null events, but must be checked separately for the usual right-continuous
% augmentation. The approximation is in $L^p(\mathbb P\otimes\lambda)$; it
% does not assert uniform-in-time convergence or convergence at the endpoints.
% }

\begin{proof}
  By Theorem~\ref{thm:banach-elem-approx} we may approximate $X$ in
  $L^p(\Omega \times [0,T];\Y)$ with error at most $\varepsilon/3$ by a
  process of the form
  \begin{equation}\label{eq:indicator}
    \widehat{X}(\omega,t)
    =\sum_{j=1}^{N}\sum_{i=0}^{m-1} \1_{I_{i}}(t)\,
    \1_{A_{i,j}}(\omega)\,y_{i,j},
  \end{equation}
  where $I_i=(s_i,s_{i+1}]$, $0=s_0<\cdots<s_m=T$,
  $A_{i,j}\in\F_{s_i}$, and $y_{i,j}\in\Y$. We omit the value at $t=0$,
  which does not affect the error. By the hypothesis on $\F_{s_i}$ and the
  countable generation of observations on $[0,s_i]$ noted above,
  Lemma~\ref{lem:smooth-Xn} approximates each $\1_{A_{i,j}}$
  arbitrarily closely in scalar $L^p$ by
  \[
    h_{i,j}(U_{q_{i,j,1}},\dots,U_{q_{i,j,k_{i,j}}}),
    \qquad 0\le q_{i,j,1}<\cdots<q_{i,j,k_{i,j}}\le s_i,
  \]
  with $h_{i,j}\in C_c^\infty(\R^{k_{i,j}})$ and $0\le h_{i,j}\le1$.
\\[4pt]
  To justify the corresponding metric $L^p$ approximation, define
  \[
    L_i(a_1,\dots,a_N):=\sum_{j=1}^N a_jy_{i,j},
    \qquad (a_1,\dots,a_N)\in[0,1]^N.
  \]
  Each $L_i$ is uniformly continuous on this cube, and
  $K:=\bigcup_{i=0}^{m-1}L_i([0,1]^N)$ is compact and has finite
  $d$-diameter. Thus scalar convergence of the coefficients in probability
  implies convergence of their images under $L_i$ in probability, while
  boundedness of the metric errors on $K$ upgrades this to $L^p$
  convergence. Consequently, substituting the functions $h_{i,j}$ into
  \eqref{eq:indicator} gives arbitrarily small metric $L^p$ error.
\\[4pt]
  Refine the partition by including all observation times $q_{i,j,\ell}$,
  and write it as $0=t_0<\cdots<t_n=T$. On a refined interval
  $(t_r,t_{r+1}]\subset I_i$, all coordinates used by $h_{i,j}$ occur among
  $U_{t_0},\dots,U_{t_r}$. To obtain compact support in all $r+1$
  coordinates, multiply $h_{i,j}$, regarded as a function of its selected
  coordinates, by a smooth cutoff in each unused coordinate. Choose these
  cutoffs between $0$ and $1$ and equal to $1$ on $[-R,R]$. The resulting
  $f_{r,j}$ belongs to $C_c^\infty(\R^{r+1})$ and satisfies
  $0\le f_{r,j}\le1$. All these replacements agree with the original
  cylinder coefficients on
  \[
    \left\{\max_{0\le r\le n}|U_{t_r}|\le R\right\},
  \]
  whose probability tends to $1$ as $R\to\infty$. Since the process values
  remain in $K$, the additional metric $L^p$ error tends to zero. Choose
  the scalar approximations and then $R$ so that their combined error is
  at most $\varepsilon/3$. Reindex the values $y_{i,j}$ on the refined
  partition and define the value at $0$ as in \eqref{eq:smooth-process}.
  This gives the required form and error at most $2\varepsilon/3$.
\\[4pt]
  Finally, omit the term at $0$. For each interval $(t_i,t_{i+1})$, convolve
  $\1_{[t_i+\delta,t_{i+1}-\delta]}$ with a nonnegative, unit-mass smooth
  mollifier supported in $(-\delta/2,\delta/2)$, for sufficiently small
  $\delta>0$. The resulting $\varphi_i\in C_c^\infty(\R)$ satisfy
  $0\le\varphi_i\le1$ and $\operatorname{supp}(\varphi_i)\subset(t_i,t_{i+1})$,
  and converge to $\1_{(t_i,t_{i+1}]}$ in $L^p([0,T])$ as $\delta\downarrow0$
  by dominated convergence. Multiplying
  the coefficient vectors in $[0,1]^N$ by the resulting $\varphi_i(t)$
  keeps them in the same cube. Uniform continuity of the maps $L_i$ and
  boundedness on $K$, as above, show that the additional metric $L^p$
  error can be made at most $\varepsilon/3$. Each coefficient depends
  only on observations up to $t_i$, and $\varphi_i$ is supported strictly
  after $t_i$, so the smoothed process remains progressive. The triangle
  inequality completes the proof.
\end{proof}

\section{Applications in non-Markovian control with random coefficients, path-dependence, and state and control distributions}
\label{sect:nonmarkovian-control}

 Throughout this section, we assume the continuum hypothesis (CH).
The preceding approximations can be used to approximate an entire control
problem, including its optimal value. We give a formulation allowing random
coefficients, dependence on the observed state history, and dependence on both
the state and control distributions. Related simplified coefficient 
approximations occur
in the stochastic HJB analysis of \cite[Lemma~5.5]{Qiu2018} and in the
jump setting of \cite{liang2026viscosity}. 
Here we establish the required
 approximations for Hilbert-space-valued jump processes, with
c\`adl\`ag state histories, possibly unbounded coefficients, and a possibly
noncompact control space.

Let $(\Omega,\F,(\F_t)_{0\le t\le T},\mathbb P)$ satisfy the usual
conditions. Let $H$ and $H_0$ be real separable Hilbert spaces, possibly
infinite dimensional.  
The stochastic basis carries a cylindrical
Brownian motion $W$ on $H_0$ and an independent
Poisson random measure $N(dt,de)$ on a standard Borel mark space
$\mathsf E$, with compensator $dt\,\vartheta(de)$, where $\vartheta$ is
$\sigma$-finite. Write $\widetilde N=N-dt\,\vartheta$.
Let $(A,\rho)$ be a nonempty Polish control space
 with a fixed compatible complete metric $\rho$, fix $a_*\in A$, and set
\[
 \mathsf D=D([0,T];H),\qquad
 \mathsf Z=H\times\mathsf D\times\mathcal P_2(H)
             \times A\times\mathcal P_2(A).
\]
For measurability, equip $\mathsf D$ with the Borel $\sigma$-algebra of
the Skorokhod $J_1$ topology. This is a Polish path space, and its Borel
$\sigma$-algebra is generated by coordinate evaluations; see
\cite[Theorem~1.7.2]{Kolokoltsov2011}.
Equip the probability-measure factors with their $W_2$-Borel
$\sigma$-algebras. Denote by $\mathcal B_{\mathsf Z}$
the resulting product $\sigma$-algebra. For estimates, however, use
\begin{equation}\label{eq:control-cadlag-metric}
 \begin{split}
 d_{\mathsf Z}(z,z')={}&\|x-x'\|_H+\|\phi-\phi'\|_\infty
       +W_2(\mu,\mu')+\rho(a,a')+W_2(\nu,\nu'),\\
 &z=(x,\phi,\mu,a,\nu),\qquad z'=(x',\phi',\mu',a',\nu').
 \end{split}
\end{equation}
Note that the supremum metric on $\mathsf D$ is generally
nonseparable, e.g., for any unit vector $v\in H$, the paths
$v\1_{[r,T]}$, $0<r<T$, have pairwise distance one.
 Nevertheless, the map
$(\phi,\psi)\mapsto\|\phi-\psi\|_\infty$ is jointly measurable for the
product of the Skorokhod Borel $\sigma$-algebras.
% , since
% \[
%  \|\phi-\psi\|_\infty
%  =\sup_{q\in(\mathbb Q\cap[0,T))\cup\{T\}}
%        \|\phi(q)-\psi(q)\|_H.
% \]
% The terminal time $T$ is included to capture a possible terminal jump.

Use the base point $z_*=(0,0,\delta_0,a_*,\delta_{a_*})$.
In addition to $Q_t(x,\phi,\mu,a,\nu)=(x,\phi_{\cdot\wedge t},\mu,a,\nu)$,
define the pre-jump stopping map by
\[
 \phi^{t-}(s)=
 \begin{cases}
  \phi(s),&s<t,\\
  \phi(t-),&s\ge t,
 \end{cases}
 \qquad Q_t^-(x,\phi,\mu,a,\nu)=(x,\phi^{t-},\mu,a,\nu),
\]
where $\phi^{0-}$ is the constant path $\phi(0)$.
Both stopping maps are contractions for $d_{\mathsf Z}$ and fixed $z_*$.
 The maps $(t,z)\mapsto Q_tz$ and $(t,z)\mapsto Q_t^-z$ are
also jointly measurable for the stated product $\sigma$-algebras.
For terminal costs use $\mathsf Z_T=H\times\mathsf D\times\mathcal P_2(H)$,
with the corresponding product $\sigma$-algebra, sum metric (as in \eqref{eq:control-cadlag-metric}), and base point
$(0,0,\delta_0)$.

 To allow linear growth while controlling Lipschitz constants, for a Banach target $B$ define
\begin{equation}\label{eq:control-weighted-norm}
 \|h\|_*
 :=\sup_{z\in\mathsf Z}\frac{|h(z)|_B}{1+d_{\mathsf Z}(z,z_*)},
 \qquad
 \operatorname{Lip}(h)
 :=\sup_{z\ne z'}\frac{|h(z)-h(z')|_B}{d_{\mathsf Z}(z,z')}.
\end{equation}
The quantity $\|h\|_*$ is a weighted supremum norm on the space of
$\mathcal B_{\mathsf Z}$-measurable, globally $d_{\mathsf Z}$-Lipschitz maps
$\mathsf Z\to B$, whereas $\operatorname{Lip}(h)$ denotes the Lipschitz
seminorm. Boundedness of $h$ is not required. In particular,
$|h(z)|_B\le\|h\|_*(1+d_{\mathsf Z}(z,z_*))$.
Denote this normed function space by
$\mathscr L_*(\mathsf Z;B)$ and equip it with the Borel $\sigma$-algebra
induced by $\|\cdot\|_*$.
The analogous notation will be used on $\mathsf Z_T$.
We impose the following assumptions.
\begin{enumerate}[(i)]
 \item The initial state $\xi$ is $\F_0$-measurable with
 $\E\|\xi\|_H^4<\infty$.
 The admissible set $\mathcal U$ is a fixed nonempty
 class of predictable $A$-valued controls satisfying for some $K\in(0,\infty)$,
 \begin{equation}\label{eq:control-moment-budget}
  \sup_{\alpha\in\mathcal U}\E\int_0^T\rho(\alpha_t,a_*)^4\,dt
  \le K.
 \end{equation}
 For bounded $A$, this allows all predictable controls.
 \item Set
 \[
  \mathsf J=L^2(\mathsf E,\vartheta;H)\cap L^4(\mathsf E,\vartheta;H),
  \qquad \|k\|_{\mathsf J}=\|k\|_{L^2(\vartheta;H)}+\|k\|_{L^4(\vartheta;H)}.
 \]
 The coefficient field $F=(b,\sigma,\gamma,f)$ takes values in the
 separable Banach space
 $B=H\times\mathcal L_2(H_0,H)\times\mathsf J\times\R$, with the sum norm;
 $\mathcal L_2(H_0,H)$ denotes the Hilbert--Schmidt operators.
 Its sections are $\mathcal B_{\mathsf Z}$-measurable and satisfy
 $F(\omega,t,z)=F(\omega,t,Q_tz)$ and
 $\operatorname{Lip}(F(\omega,t,\cdot))\le L$, outside a common
 $\mathbb P\otimes dt$-null set in the predictable $\sigma$-algebra
 $\mathcal P$.
 The map $(\omega,t)\mapsto F(\omega,t,\cdot)$ is
 $\mathcal P/\mathcal B(\mathscr L_*(\mathsf Z;E))$-measurable,
 where the target Borel $\sigma$-algebra is induced by
 \eqref{eq:control-weighted-norm}, and
 \[
  \E\int_0^T\|F(\omega,t,\cdot)\|_*^4\,dt<\infty.
 \]
 %We use a measurable version, setting it equal to zero on exceptional sets.
 \item The terminal field $g(\omega,\cdot):\mathsf Z_T\to\R$ has
 measurable sections for the stated product $\sigma$-algebra, is
  $\F_T/\mathcal B(\mathscr L_*(\mathsf Z_T;\R))$-measurable
 in the analogous weighted norm, has Lipschitz constant
 at most $L$ almost surely, and satisfies $\E\|g\|_*^4<\infty$.
\end{enumerate}

The function-space measurability assumptions are substantive and are
distinct from measurability of the parameter sections.
Since $\mathsf Z$ and $\mathsf Z_T$ are standard Borel spaces and
$B$ is separable, the two coefficient function spaces have cardinality,
and hence density, at most $\mathfrak c$. Under CH,
Theorem~\ref{thm:second} therefore implies that $F$ and $g$ are strongly
measurable with respect to $\mathbb P\otimes dt$ and $\mathbb P$,
respectively.
Alternatively, assume pointwise predictable measurability and an
essentially separable range in $\|\cdot\|_*$. Choose a countable dense
family in the difference set of that range and, for each member, a
sequence of parameter tuples approaching its weighted supremum.
Their union is a countable set that determines the weighted norm of
every such difference, by density. The argument of
Theorem~\ref{thm:measurable-function-valued} then gives a strongly
measurable version. This uses a set of evaluations chosen for the
particular range, rather than a countable supremum-metric dense set
in $\mathsf D$. Pointwise measurability alone is insufficient.

For $\alpha\in\mathcal U$, consider the controlled path-dependent McKean--Vlasov equation
\begin{equation}\label{eq:nonmarkovian-controlled-sde}
 \begin{aligned}
  dX_t^\alpha
   &=b(\omega,t,Z_t^\alpha)\,dt
     +\sigma(\omega,t,Z_t^\alpha)\,dW_t
  +\int_{\mathsf E}\gamma(\omega,t,Z_t^\alpha,e)\,
                      \widetilde N(dt,de),\qquad X_0^\alpha=\xi,\\
  Z_t^\alpha
   &=\bigl(X_{t-}^\alpha,(X^\alpha)^{t-},
           \mu_t^\alpha,\alpha_t,\nu_t^\alpha\bigr),\\
  \mu_t^\alpha&=\mathcal L(X_{t-}^\alpha),\qquad
  \nu_t^\alpha=\mathcal L(\alpha_t).
 \end{aligned}
\end{equation}
Here $\mathcal L$ denotes the unconditional law under $\mathbb P$, and
$X_{0-}^\alpha=\xi$. All stochastic integrands use pre-jump states and
histories. In the drift and running cost, replacing these by the current
state and stopped history gives the same time integrals, since a
c\`adl\`ag path has only countably many jumps. The corresponding state
laws agree for Lebesgue-almost every time. The ordinary identity
$F=F\circ Q_t$ is not being identified with the stronger identity
$F=F\circ Q_t^-$ on arbitrary parameter paths.
The control law belongs to $\mathcal P_2(A)$ for almost every $t$;
at the deterministic exceptional times it may be defined as $\delta_{a_*}$.
Choose its Borel measurable version as a measure-valued time map.
The running cost $f$ has precisely the same arguments as $b$ and $\sigma$.
Define
\begin{equation}\label{eq:nonmarkovian-control-value}
 \begin{split}
 J(\alpha)&=\E\left[\int_0^T f(\omega,t,Z_t^\alpha)\,dt
       +g(\omega,X_T^\alpha,X^\alpha,\mathcal L(X_T^\alpha))\right],\\
 V&=\inf_{\alpha\in\mathcal U}J(\alpha).
 \end{split}
\end{equation}
Thus the drift, diffusion, jump amplitude, and running cost may depend
on all the prescribed random, state, history, law, and control variables,
with the predictable pre-jump convention for the stochastic integrals.
The diffusion may be controlled and degenerate. The intensity
$\vartheta(\mathsf E)$ may be infinite; the assumptions on $\mathsf J$
control the jump integrals. There is no Markov assumption on the
random environment or the state.

\begin{proposition}[Approximation of the control problem]
\label{prop:nonmarkovian-control-stability}
Under assumptions (i)--(iii), equation
\eqref{eq:nonmarkovian-controlled-sde} has a unique strong solution for
each $\alpha\in\mathcal U$, with c\`adl\`ag $H$-valued paths. There exist
elementary function-valued fields
$\bar F^n=(\bar b^n,\bar\sigma^n,\bar\gamma^n,\bar f^n)$ and simple terminal fields
$g^n$, whose deterministic sections all have Lipschitz constant at most $L$,
such that
\begin{equation}\label{eq:control-coefficient-error}
 \eta_n:=\left(\E\int_0^T\|\bar F^n-F\|_*^4\,dt\right)^{1/4}
             +\left(\E\|g^n-g\|_*^4\right)^{1/4}\longrightarrow0.
\end{equation}
Set $F^n(\omega,t,z)=\bar F^n(\omega,t,Q_t^-z)$ and define
$X^{n,\alpha},J_n(\alpha),V_n$ by replacing $(F,g)$ with $(F^n,g^n)$
in \eqref{eq:nonmarkovian-controlled-sde}--\eqref{eq:nonmarkovian-control-value}.
In the approximating equation, the state law is
$\mu_t^{n,\alpha}=\mathcal L(X_{t-}^{n,\alpha})$, and the terminal
cost uses $\mathcal L(X_T^{n,\alpha})$.
There is a constant $C$, independent of $n$ and $\alpha$, such that
\begin{equation}\label{eq:control-stability-bound}
 \sup_{\alpha\in\mathcal U}
   \left(\E\sup_{0\le t\le T}\|X_t^{n,\alpha}-X_t^\alpha\|_H^2\right)^{1/2}
 +\sup_{\alpha\in\mathcal U}|J_n(\alpha)-J(\alpha)|
 +|V_n-V|\le C\eta_n.
\end{equation}
Consequently, every $\varepsilon$-optimal control for $J_n$ is
$(\varepsilon+2C\eta_n)$-optimal for $J$.
\end{proposition}

\begin{proof}
The measurability gives strong measurability as well as an essentially separable range for $F$ in
$\|\cdot\|_*$. Since predictable processes are progressive, apply
Theorem~\ref{thm:banach-elem-approx} with $p=4$
to a measurable version, using the closure of this separable range inside
the space of maps with Lipschitz constant at most $L$ as the target.
This target has countable density. Its sections remain
$\mathcal B_{\mathsf Z}$-measurable and $L$-Lipschitz: convergence in
the weighted norm implies pointwise convergence. Apply Lemma~\ref{lem:approx} similarly
to $g$. These applications give \eqref{eq:control-coefficient-error}
and preserve the Lipschitz bound by selecting the deterministic sections
from the corresponding targets. Set $F^-(\omega,t,z)=F(\omega,t,Q_t^-z)$.
Since $Q_t^-z_*=z_*$,
\[
 \|h\circ Q_t^-\|_*\le\|h\|_*,\qquad e_n:=\|\bar F^n-F\|_*,
\]
and
\[
 |F^n(\omega,t,z)-F^-(\omega,t,z)|_B
       \le e_n(\omega,t)(1+d_{\mathsf Z}(z,z_*)).
\]
For every c\`adl\`ag path $\phi$, the map $t\mapsto\phi^{t-}$ is
left-continuous in the supremum metric on $(0,T]$, with right limit
$\phi_{\cdot\wedge t}$. Coordinate measurability implies that
$(t,\phi)\mapsto\phi^{t-}$ is measurable for the Skorokhod Borel
structure. Consequently, for any adapted c\`adl\`ag process $X$,
$X^{t-}$ is an adapted left-continuous Skorokhod-space-valued process,
hence predictable. The law of $X_{t-}$ is a left-continuous
$\mathcal P_2(H)$-valued map when $\E\sup_t\|X_t\|_H^2<\infty$.
These facts make $Z^\alpha$ predictable.
Evaluation of predictable simple fields with measurable parameter
sections at $Z^\alpha$ is predictable. Approximation in $\|\cdot\|_*$
therefore proves the same assertion for $F$; it also applies directly
to $F^n$. The resulting $\mathsf J$-valued jump integrands admit jointly
predictable representatives in $(\omega,t,e)$, which we use in the
Poisson integrals. This argument does not require measurability of the
random path for the Borel $\sigma$-algebra of the supremum topology.

The usual Picard argument applies also with the distribution argument:
on the common probability space the coupling inequality is
\begin{equation}\label{eq:control-law-coupling}
 W_2\bigl(\mathcal L(Y_t),\mathcal L(Y'_t)\bigr)^2
       \le\E\|Y_t-Y'_t\|_H^2.
\end{equation}
Together with the Lipschitz bound in the stopped-path supremum metric,
this gives uniqueness and convergence of Picard iteration on successive
short time intervals, using also the same inequality at $t-$.
For completeness, the additional jump estimates, for predictable
$H$-valued integrands $k$, are
\begin{equation}\label{eq:control-jump-maximal}
 \begin{split}
 \E\sup_{t\le T}\left\|\int_0^t\int_{\mathsf E}k_s(e)\,
                        \widetilde N(ds,de)\right\|_H^2
 &\le C\,\E\int_0^T\|k_s\|_{L^2(\vartheta;H)}^2\,ds,\\
 \E\sup_{t\le T}\left\|\int_0^t\int_{\mathsf E}k_s(e)\,
                        \widetilde N(ds,de)\right\|_H^4
 &\le C\,\E\left(\int_0^T\|k_s\|_{L^2(\vartheta;H)}^2\,ds\right)^2
%  \\
%  &\quad
 +C\,\E\int_0^T\|k_s\|_{L^4(\vartheta;H)}^4\,ds .
 \end{split}
\end{equation}
See \cite[Theorem~3.2]{MarinelliRockner2013}. This explains the use of
both $L^2(\vartheta;H)$ and $L^4(\vartheta;H)$ in the coefficient target.
Set $a_t:=|F(\omega,t,z_*)|_B\le\|F(\omega,t,\cdot)\|_*$.
The uniform Lipschitz bound gives the additive growth estimate
\[
 |F(\omega,t,z)|_B\le a_t+L\,d_{\mathsf Z}(z,z_*).
\]
For $M_\alpha(t):=\E\sup_{r\le t}\|X_r^\alpha\|_H^4$, the preceding jump
estimates, the Hilbert-space Brownian maximal inequality, and Jensen's
inequality yield
\[
 \begin{aligned}
 M_\alpha(t)\le C_{L,T}\bigg[&\E\|\xi\|_H^4
 +\E\int_0^t\bigl(a_s^4+\rho(\alpha_s,a_*)^4\bigr)\,ds
 +\int_0^t M_\alpha(s)\,ds\bigg],
 \end{aligned}
\]
where $C_{L,T}$ is independent of $\alpha$. The moment bounds are first
established uniformly for the Picard iterates and passed to the solution
by Fatou's lemma. Gronwall's inequality and the control moment bound
\eqref{eq:control-moment-budget} therefore give
\begin{equation}\label{eq:control-fourth-moments}
 \sup_{\alpha\in\mathcal U}\E\sup_{t\le T}\|X_t^\alpha\|_H^4<\infty,
 \qquad
 \sup_{\alpha\in\mathcal U}\E\int_0^T(R_t^\alpha)^4\,dt<\infty,
 \quad 
 R_t^\alpha:=1+d_{\mathsf Z}(Z_t^\alpha,z_*).
\end{equation}
Indeed, the squares of the second moments of the state and control
laws are bounded by the corresponding fourth moments, by Jensen's
inequality. Assumption \eqref{eq:control-moment-budget} controls the
control terms. The same existence argument applies to each approximating
equation. Only the original moment bounds in
\eqref{eq:control-fourth-moments} are needed for the error estimate.

Evaluate the coefficient error at the original state and apply
H\"older's inequality on $\Omega\times[0,T]$. Since $Q_t^-Z_t^\alpha=Z_t^\alpha$,
$F^-(\omega,t,Z_t^\alpha)=F(\omega,t,Z_t^\alpha)$, so
\begin{equation}\label{eq:control-weighted-error}
 \sup_{\alpha\in\mathcal U}
 \E\int_0^T|F^n(\omega,t,Z_t^\alpha)-F(\omega,t,Z_t^\alpha)|_E^2\,dt
 \le\left(\E\int_0^T e_n^4\,dt\right)^{1/2}
       \sup_{\alpha\in\mathcal U}
          \left(\E\int_0^T(R_t^\alpha)^4\,dt\right)^{1/2}
 \le C\eta_n^2.
\end{equation}
For the remaining difference use the common Lipschitz constant of $F^n$
and \eqref{eq:control-law-coupling}. The control and its law are identical
in the two equations for each fixed $\alpha$. Writing
$u_n^\alpha(t)=\E\sup_{s\le t}\|X_s^{n,\alpha}-X_s^\alpha\|_H^2$,
the drift, Brownian, and jump estimates yield
\[
 u_n^\alpha(t)\le C\eta_n^2+C\int_0^t u_n^\alpha(s)\,ds.
\]
Gronwall's inequality proves the state estimate uniformly in $\alpha$.
The Lipschitz bounds for $f^n,g^n$,
\eqref{eq:control-weighted-error}, and its terminal counterpart prove
$\sup_\alpha|J_n(\alpha)-J(\alpha)|\le C\eta_n$.
Here the terminal weight is
$1+\|X_T^\alpha\|_H+\|X^\alpha\|_\infty+
W_2(\mathcal L(X_T^\alpha),\delta_0)$, whose fourth moment is uniformly bounded.
The costs and their infima are finite by the same estimates. Finally,
\[
 |\inf_{\alpha\in\mathcal U}J_n(\alpha)
       -\inf_{\alpha\in\mathcal U}J(\alpha)|
 \le\sup_{\alpha\in\mathcal U}|J_n(\alpha)-J(\alpha)|.
\]
If $J_n(\alpha_n)\le V_n+\varepsilon$, it follows that
$J(\alpha_n)\le V+\varepsilon+2C\eta_n$, as claimed.
\end{proof}

\begin{remark}[Smooth dependence on finitely many observations]
\label{rem:control-smooth-cylinders}
Suppose additionally that the filtration is generated, modulo null events
at every time, by a countable family of real-valued processes
$(U^k)_{k\ge1}$, each with left- or right-continuous paths.
This includes the generating process in
Theorem~\ref{thm:smooth-process-approx}, and separable-Hilbert-valued
generating processes through their coordinates in a countable orthonormal
basis. On each $[0,t]$, their observations at rational times together
with $0,t$ form a countable generating family, so the proof of
Lemma~\ref{lem:smooth-Xn} applies without change.
The fields $\bar F^n$ and $g^n$ can then be replaced by fields whose
randomness depends smoothly on finitely many of these scalar observations, with
the same convergence conclusion. One must preserve a uniform Lipschitz
constant when smoothing. On each elementary interval, let
$h_1,\ldots,h_N$ be the deterministic coefficient sections and approximate
their indicator weights by smooth cylinder functions
$f_1,\ldots,f_N\in[0,1]$ on a common finite set of past observations.
For $\delta>0$, use the weights
\[
 w_j=\frac{f_j}{\delta+\sum_{k=1}^N f_k},\qquad
 \sum_{j=1}^N w_j\le1.
\]
They are smooth and compactly supported when the $f_j$ are, and
$\sum_jw_jh_j$ is a convex combination of the $h_j$ and the zero field.
Its Lipschitz constant is at most $L$. As the scalar approximation
errors and then $\delta$ tend to zero, these coefficients converge to
the elementary field in $L^4$ of $\|\cdot\|_*$: for each finite
construction their norms are bounded by a deterministic constant, and
the weights converge in probability to the indicator partition.
The same construction applies to the terminal field, using observations
up to $T$. Time indicators may also be smoothed inside their respective
intervals, as in Theorem~\ref{thm:smooth-process-approx}, without increasing
the Lipschitz bound. Compose with $Q_t^-$ after these operations to
obtain coefficients evaluated on pre-jump histories. Their evaluation
along a controlled state is predictable. This gives smooth dependence on the observation
coordinates; it does not assert differentiability in the state-path or
measure variables, or time smoothness after the stopped-path composition.
\end{remark}

\begin{remark}[Conditional laws and joint state-control laws]
\label{rem:control-law-extensions}
The proposition also holds with
$\mu_t^\alpha=\mathcal L(X_{t-}^\alpha\mid\mathcal F_{t-}^0)$ and
$\nu_t^\alpha=\mathcal L(\alpha_t\mid\mathcal F_{t-}^0)$ for a fixed
common-noise subfiltration $(\mathcal F_t^0)\subseteq(\mathcal F_t)$,
provided regular conditional distributions with jointly predictable
versions are available throughout the construction. On a common
probability space, the conditional law of $(Y_t,Y'_t)$ is a coupling, so
\[
 \E W_2\bigl(\mathcal L(Y_t\mid\mathcal F_{t-}^0),
             \mathcal L(Y'_t\mid\mathcal F_{t-}^0)\bigr)^2
       \le\E\|Y_t-Y'_t\|_H^2.
\]
Conditional Jensen's inequality gives the fourth-moment estimates used
above, and the same proof applies, with
$\mathcal L(X_T^\alpha\mid\mathcal F_T^0)$ in the terminal cost.
At $t=0$, use $\mathcal F_0^0$ in place of $\mathcal F_{0-}^0$.
The common-noise filtration is kept
fixed when comparing controls and coefficient approximations.
One may also replace the two marginal-law arguments by the joint law
$\mathcal L(X_{t-}^\alpha,\alpha_t)$, or its conditional version, using
$W_2$ on $H\times A$ with the sum metric and imposing the analogous
field assumptions. Coupling the same control in both systems gives
\eqref{eq:control-law-coupling} for the unconditional joint laws
and the preceding conditional coupling estimate for their
conditional versions.
\end{remark}

\begin{remark}[Scope of the control approximation]
The fourth-moment assumptions above provide a convenient sufficient
condition for uniformity over a noncompact class of controls. The
coefficient approximation is in an integrated function-space norm;
it is not an essential-supremum approximation in $(\omega,t)$.
The conclusion concerns controlled states, costs, optimal values, and
transfer of approximately optimal controls. It does not by itself give
a dynamic programming principle or a viscosity characterization for
the path- and law-dependent problem. In particular, pasting controls
on an event can change their unconditional laws and hence coefficients
outside that event. The finite-dimensional stochastic HJB arguments of
\cite{Qiu2018,liang2026viscosity} therefore require a separate extension
to a suitable path/distribution state and an appropriate admissible
class. Finite observation approximations of the random environment
still retain the state-history and distribution arguments. Additional
path or measure differentiability requires further assumptions. The
present statement covers separable-Hilbert-valued SDEs with a fixed
reference Poisson compensator; arbitrary nonseparable Hilbert spaces,
unbounded drift generators in mild or variational SPDE formulations,
and more general controlled compensators require their own hypotheses
and stability arguments.
\end{remark}

The preceding construction may be viewed as a \textit{piecewise Markovian-type
coefficient approximation}. On each elementary interval $(t_i,t_{i+1}]$,
deterministic sections are selected by finitely many weights measurable
with respect to $\F_{t_i}$, with the history and distribution arguments retained as enlarged
state variables. This terminology describes the coefficient structure and
does not require the state process under every admissible open-loop control
to be Markov. A formulation with finitely many history coordinates requires
a separately justified reduction of the path dependence and sufficient
environmental variables to determine subsequent evolution; unrestricted
distribution dependence generally leaves the Bellman state infinite
dimensional even when $H=\R^d$. Related finite-observation and
state-dependent approximations are used in the Bellman constructions of
\cite[Lemma~5.5 and the proof of Theorem~5.6]{Qiu2018} and
\cite{TangZhang2015}, while \cite{Tan2014} gives a discrete-history Bellman
recursion and convergence of approximating values. For continuous
path- and distribution-dependent dynamics with deterministic coefficient
functions, \cite{CossoEtAl2023} establishes dynamic programming and the
viscosity property of the associated Master Bellman equation, including
control-law dependence and Hilbert-valued states; \cite{PhamWei2018}
develops Bellman equations and viscosity comparison for dependence on the
joint state-control law under a feedback formulation. The more recent
preprint \cite{ZhouTouziZhang2025} treats progressive random coefficients
and path-dependent process arguments, including common-noise mean-field
control, through a viscosity theory on process space. Complementary results
with jumps include the Bellman equation and smooth verification theorem for
state-law-dependent jump diffusions in \cite{GuoPhamWei2023}, and
stochastic HJB equations with random coefficients, classical verification,
and Sobolev well-posedness under additional restrictions in
\cite{MengEtAl2023}. These works treat different combinations of the
features considered here under their respective hypotheses. Applying their
Bellman theories to the approximating problems therefore requires a
suitable dynamic programming state and admissible control class beyond
the stability conclusion of
Proposition~\ref{prop:nonmarkovian-control-stability}.

{\footnotesize
% Redefine thebibliography to hook in a tighter item separation
\let\oldthebibliography\thebibliography
\renewcommand\thebibliography[1]{%
    \oldthebibliography{#1}%
    \setlength{\itemsep}{0pt}%  <-- Reduces space between items
    \setlength{\parskip}{0pt}%  <-- Reduces paragraph space
}
\bibliographystyle{siam}
\bibliography{references}
}

\end{document}